\documentclass[12pt]{amsart}

\usepackage{amssymb,latexsym,enumitem}
\usepackage{pgf,tikz,ifthen,arrayjobx}
\usetikzlibrary{arrows}
\usetikzlibrary{calc}

\usepackage{wasysym} 
\usepackage[breaklinks=true,colorlinks=true,linkcolor=green,citecolor=red,urlcolor=blue]{hyperref}

\allowdisplaybreaks

\makeatletter
\@namedef{subjclassname@2020}{%
  \textup{2020} Mathematics Subject Classification}
\makeatother

\theoremstyle{plain}
\newtheorem{theorem}{Theorem}
\newtheorem{corollary}{Corollary}
\newtheorem{proposition}{Proposition}
\newtheorem{conjecture}{Conjecture}
\newtheorem{lemma}{Lemma}
\newtheorem{Definition}{Definition}
\newtheorem*{theorem*}{Theorem}
\newtheorem*{corollary*}{Corollary}

\theoremstyle{definition}
\newtheorem{remark}{Remark}

\newcounter{HConditions}

\newcommand{\Sa}{\overline{\mathcal{S}}}
\newcommand{\vh}{v_{\makebox[0.3pt][l]{\hexstar}\hexagon}}
\newcommand{\oG}{\overline{G}}

\begin{document}

\title{$K_{1,2,2,2}$ has no $n$-fold planar cover graph for $n<14$}

\author{Dickson Y. B. Annor}


\thanks{Department of Mathematical and Physical Sciences, La Trobe University, Bendigo 3552 and Bundoora 3086, Australia}

\author{Yuri Nikolayevsky}

\author{Michael S. Payne}



\thanks{Corresponding author's email: d.annor@latrobe.edu.au}

\thanks{Dickson Annor was supported by a La Trobe Graduate Research Scholarship.}
\thanks{Michael Payne was partially supported by a
DECRA from the Australian Research Council.}

\begin{abstract}%
 S. Negami conjectured in $1988$ that a connected graph has a finite planar cover if and only if it embeds in the projective plane. It follows from the works of D. Archdeacon, M. Fellows, P. Hlin\v{e}n\'{y}, and S. Negami that this conjecture is true if the graph $K_{1, 2, 2, 2}$ has no finite planar cover. We prove a number of structural results about putative finite planar covers of $K_{1,2,2,2}$ that may be of independent interest.  We then apply these results to prove that $K_{1, 2, 2, 2}$ has no planar cover of fold number less than $14$.
\end{abstract}

\maketitle

\section{Introduction}\label{sec:intr}
In this paper, we deal with finite simple undirected graphs. For a graph $K$, let $V(K)$ and $E(K)$ respectively denote the vertex set and edge set of $K$. A graph $G$ is a \textit{cover} of a graph $K$ if there exists an onto mapping $\pi : V(G) \to V(K)$, called a (cover) \textit{projection}, such that $\pi$ maps the neighbours of any vertex $v$ in $G$ bijectively onto the neighbours of $\pi(v)$ in $K$. A cover is called \textit{planar} if it is a planar graph. Note that every planar graph has a planar cover by the identity projection, but there are also non-planar graphs having planar covers.  If $K$ is connected, then $|\pi^{-1}(v)| = n$ is the same for all $v \in V(K)$ and $\pi$ is called an $n$-fold cover.
 
In $1988$  S. Negami made the following conjecture.
 \begin{conjecture}[Negami's Conjecture~\cite{8}]\label{conj: negami} 
      A connected graph has a finite planar cover if and only if it embeds in the projective plane.
 \end{conjecture} 
 If a graph embeds in the projective plane, then the lift of this embedding to the sphere is a planar $2$-fold cover. Thus every projective-planar graph has a planar cover. The converse is not true for disconnected graphs. For example, the graph consisting of two disjoint copies of $K_5$ has a planar cover, but does not embed in the projective plane.
 
 In the years after Conjecture~\ref{conj: negami} was posed, the results of  D. Archdeacon \cite{1}, M. Fellows \cite{4}, P. Hlin\v{e}n\'{y} \cite{5} and \cite{7}, and S. Negami \cite{10}  combined to show that it is equivalent to the following statement \textit{`$K_{1, 2, 2, 2}$ has no finite planar cover'}. The graph $K_{1, 2, 2, 2}$ consists of the octahedron with an apex vertex connected to all other vertices, see Figure~\ref{fig:K2221}. Furthermore, Archdeacon and Richter~\cite{3} showed that for any planar cover of a non-planar graph, the fold number is even. It is also known that $K_{1, 2, 2, 2}$ has no planar $2$-fold cover \cite{9}. For other related partial results on Negami's conjecture, we refer interested readers to \cite{6}.

\begin{figure}[ht]
    \centering
    \begin{tikzpicture}[scale=0.375]
 \draw[ thick ] (0,0)--(8,0); 
 \draw[ thick ] (0,0)--(4,8); 
 \draw[ thick ] (8,0)--(4,8); 
 \draw[ thick ] (2,4)--(6,4); 
 \draw[ thick ] (4,0)--(6,4); 
 \draw[ thick ] (4,0)--(2,4); 
 \draw[ thick ] (6,4)--(8,8); 
 \draw[ thick ] (2,4)--(8,8); 
 \draw[ thick ] (0,0) .. controls (-2,8) and (2, 11) .. (8, 8);
 \draw[ thick ] (0,0) .. controls (0,5) and (2, 6) .. (4, 8);
 \draw[ thick ] (8,0) .. controls (8,5) and (6, 6) .. (4, 8);
 \draw[ thick ] (0,0) .. controls (2,-1.5) and (5, -1.5) .. (8, 0);
 \draw[ thick ] (4,0) .. controls (5, 0.5) and (7.5, 1) .. (8, 8);
 \draw[ thick ] (8,0) .. controls (9.5,2) and (9.5, 5)  .. (8, 8);
 \draw[ thick ] (4,8) .. controls (6,8.5) .. (8, 8);
 \filldraw[black] (0,0) circle (6pt);
 \filldraw[black] (4,0) circle (6pt);
 \filldraw[black] (8,0) circle (6pt);
 \filldraw[black] (2,4) circle (6pt);
 \filldraw[black] (8,8) circle (6pt);
 \filldraw[black] (4,8) circle (6pt);
 \filldraw[black] (6,4) circle (6pt);
    \end{tikzpicture}
    \caption{The graph $K_{1,2, 2, 2}$.}
    \label{fig:K2221}
\end{figure}
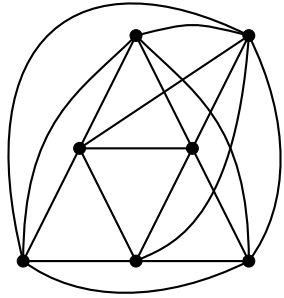

The purpose of this paper is to study the structure of planar covers of $K_{1, 2, 2, 2}$ in the hope that we might make useful progress towards a proof of Conjecture~\ref{conj: negami}. As an application of our results we prove the following. 
\begin{theorem}\label{th:nofold}
No $4, 6, 8, 10$ or $12$-fold cover of the graph $K_{1, 2, 2, 2}$ is planar.
\end{theorem}
Combined with the earlier results mentioned above this implies: 
\begin{corollary} 
If a connected graph a has planar cover of fold number less than 14, then it embeds in the projective plane.
\end{corollary}
Moreover, since any possible counterexample to Conjecture~\ref{conj: negami} would contain $K_{1,2,2,2}$ (or another excluded minor for projective planar graphs that reduces to $K_{1,2,2,2}$ via $Y\Delta$-transformations, see~\cite{6}) as a minor, its planar cover graph would have fold number at least 14 and
at least 98 vertices.

We note that N. Takahashi \cite{12}  worked on finite planar covers of $K_{1,2,2, 2}$ in  his Master's Thesis  supervised by K. Ota. He  proved that $K_{1,2, 2, 2}$ has no planar cover of fold number less than $12$. His method of proof was a form of discharging argument. In contrast, our approach is inspired by the `structural' methods employed by Archdeacon~\cite{1}.

The paper is organised as follows.
In Section~\ref{sec:prel}, we state some necessary definitions and known lemmas.
In Section~\ref{sec:cover} we prove an important lemma, that if Conjecture~\ref{conj: negami} fails there must be a planar cover $G$ of $K_{1,2,2,2}$ in which triangular facial cycles are the only facial cycles that cover triangles. 
In Section~\ref{sec:pptH}, we select a particular subgraph $G'$ of $G$, and a cover $H$ of $K_4$ inside $G'$, then go on to establish a number of useful properties of $G'$ and $H$.
In Section~\ref{s:bsn} we show that $H$ cannot have just one internal non-triangular face and in Section~\ref{s:2f} we show that $H$ cannot have just two internal non-triangular faces.
Up to this point, the results presented do not depend on the fold number of the cover $G$. In Section~\ref{sec:proofthm} we turn our attention to covers of low fold number and prove Theorem~\ref{th:nofold}.

\section{Preliminaries}\label{sec:prel}

Let graph $G$ be a cover of graph $K$. For any subgraph $K'$ of $K$, we call the graph $G' = \pi^{-1}(K')$ the \emph{lift of} $K'$ \emph{into G}. The lift of a cycle $C$ of $K$ into $G$ is a collection of disjoint cycles $C'_1, C'_2, \dots, C'_s$. Note that the length of $C$ divides the length of $C'_i$, for $i = 1, 2, \dots, s$. A cycle $C'$ in $G$ is called \textit{short} if its length is equal to the length of $C$. Otherwise, it is called \textit{long}. A cycle $C$ is \textit{peripheral} in a graph $K$ if it is chordless and $K \setminus C$ is connected.

\begin{lemma}\label{lem:shortcycle}(Archdeacon~\cite{1})
Let $G$ be a plane cover of a graph $K$. Let $C'$ be a short cycle of $G$ covering a peripheral cycle $C$ in $K$. Then $C'$ is a face boundary.
\end{lemma}

Let $G$ be a plane graph, and let $f$ be the outer face of $G$. The graph $G$ is a \textit{semi-cover} of a graph $K$ if there exists an onto mapping $\rho : V(G) \to V(K)$, called a \textit{semi-projection}, such that for each vertex $v$ of $G$ not incident with $f$, $\rho$ maps the neighbours of $v$ bijectively onto the neighbours of $\rho(v)$ in $K$, and for each vertex $w$ of $G$ incident with $f$, $\rho$ maps the neighbours of $w$ injectively to the neighbours of $\rho(w)$. It is clear that every cover is a semi-cover but the converse is not true. The notion of semi-cover was first introduced by Hlin\v{e}n\'{y} \cite{7}.

We will need the following fact about planar covers of $K$ with minimum fold number \cite[Corollary~10]{11}. 
 
\begin{lemma}\label{lem:3-cont}
Suppose graph $G$ is a minimum fold planar cover of graph $K$. Then if $K$ is $3$-connected, then $G$ is $3$-connected.
\end{lemma}

Throughout this paper, we label every vertex of $G$ with the name of the corresponding vertex in $K$ (except sometimes we add a subscript or dash in order to distinguish particular vertices). That is, each vertex $v$ in $G$ gets the same label as $\pi(v)$ in $K$. For the sake of brevity, we will often refer to a vertex of $G$ labelled $x$ simply as `a vertex $x$', or to a triangle in $G$ labelled $(a,b,c)$ as just `a triangle $(a,b,c)$'.

We label the vertices of $K_{1,2,2,2}$ as follows: the vertex of degree $6$ is labelled $0$, and the six vertices of degree $5$ are labelled $\{\pm 1, \pm 2, \pm 3\}$ in such a way that the vertices $i$ and $-i$ are not adjacent for $i=1,2,3$. Note that every vertex $\pm i$ of degree $5$ is adjacent to the vertex $0$ and to four vertices $\pm j, \pm k$, where $\{i,j,k\}=\{1,2,3\}$.

\section{Selecting a particular cover}\label{sec:cover}


Assume $K_{1,2,2,2}$ has a finite planar cover graph. We identify each cover with its plane drawing. Note that this identification is well-defined by Whitney's theorem, as the cover is 3-connected by Lemma~\ref{lem:3-cont}.  Out of all such covers we choose one with the following properties:

\begin{enumerate}[label=(\Alph*),ref=\Alph*]
    \item \label{it:minfold}
    It has the smallest fold number.

    \item \label{it:max3}
    Out of those covers satisfying~\eqref{it:minfold}, it has the maximal number of triangular faces.

    \item \label{it:shortf}
    Any short cycle covering a triangle in $K_{1,2,2,2}$ bounds a face.~(This property is guaranteed by Lemma~\ref{lem:shortcycle}.)
\end{enumerate}

\begin{lemma} 
  Out of those planar covers satisfying~\eqref{it:minfold},~\eqref{it:max3} and~\eqref{it:shortf}, there is a cover such that: 
  \begin{enumerate}[resume,label=(\Alph*),font=\normalfont,ref=\Alph*]
    \item \label{it:longf}
     No long cycle covering a triangle in $K_{1,2,2,2}$ bounds a face.
\setcounter{HConditions}{\value{enumi}}
\end{enumerate}
\end{lemma}
\begin{proof}
  Call a face bounded by a long cycle covering a triangle in $K_{1,2,2,2}$ \emph{long cyclic}. Arguing by contradiction, suppose that any cover satisfying~\eqref{it:minfold} and~\eqref{it:max3} contains a long cyclic face. Let $\mathcal{G}$ be the set of such covers having the smallest possible number of long cyclic faces. 
  
  We claim that there exists $G \in \mathcal{G}$ with all the faces but one being triangular, and with this remaining face being long cyclic. Suppose this is not true, and so any $G \in \mathcal{G}$ contains a long cyclic face $F$ (potentially, more than one) and a non-triangular face $F' \ne F$ (which may or may not also be long cyclic). Define the distance $d(F, F')$ between $F$ and $F'$ to be the smallest number $m \ge 1$ for which there exists a sequence of faces $F',F_1,F_2, \dots, F_{m-1},F$ of $G$, in which any two consecutive faces share an edge. Let $m(G) \ge 1$ be the minimum of the numbers $d(F, F')$ where $F$ is a long cyclic face of $G$ and $F' \ne F$ is a non-triangular face of~$G$. 
  
  Denote $M = \min \{m(G) \, : \, G \in \mathcal{G}\}$ (note that $M \ge 1$ by our assumption), and choose a drawing $G$, a long cyclic face $F$ of $G$ and a non-triangular face $F' \ne F$ of $G$ such that $d(F,F')=M$. Let $F',F_1,F_2, \dots, F_{M-1},F$ be a sequence of faces of $G$ that realises this distance
  ; note that $F_i$ and $F$ have no common edges for $1 \le i \le M-2$.
  
  Suppose the boundary $C$ of $F$ is the lift of a cycle $abc$ of $K_{1,2,2,2}$ and that the vertices of $C$ are labelled $abcabc\dots$ in the positive direction. Suppose $(c, a)$ is an edge shared by $F_{M-1}$ and $F$. We simultaneously replace the union of all edges $(c, a)$ of $C$ by the union of disjoint edges lying in $F$ by joining every vertex $a$ of $C$ to the next vertex $c$ in the positive direction on $C$. We get another planar cover $G'$ of $K_{1,2,2,2}$. Note that in $G'$, all the faces are still discs (by~\eqref{it:minfold}). Moreover, every face of $G$ sharing an edge $(c, a)$ with $F$ in $G$ must be triangular (note that this implies $M > 1$), for otherwise $G'$ would have had more triangular faces than $G$, in contradiction with~\eqref{it:max3}. The third vertex of any such triangular face has a label different from any of $a, b$ or $c$, and all these vertices are different vertices of $G$. If at least two of them have different labels, then $G'$ still satisfies~\eqref{it:minfold} and~\eqref{it:max3}, and has one fewer long cyclic faces than $G$, in contradiction with the fact that $G \in \mathcal{G}$. Hence all these third vertices must have the same label, so that the above procedure produces a long cyclic face $F''$ of $G'$ (and destroys the long cyclic face $F$ of $G$). As all the faces of $G$ that do not share an edge $(c, a)$ with $F$ are still faces of $G'$, we deduce that $G' \in \mathcal{G}$. By the same reason, all the faces $F', F_1, \dots, F_{M-2}$ of $G$ are faces of $G'$, with any consecutive two sharing an edge. Moreover, by our construction, the face $F_{M-2}$ of $G'$ shares an edge with $F''$ (one of the edges of the ``former" triangular face $F_{M-1}$ of $G$). It follows that the distance $d(F'', F')$ in $G$ is at most $M-1$, a contradiction.
  
  We have now constructed a cover $G \in \mathcal{G}$ with the property that one of its faces $F$ is long cyclic and all the other faces are triangular. We want to show that such $G$ cannot exist. 
  
  Denote $3m, \, m \ge 2$, the length of the cycle $C$ bounding $F$, and assume that $F$ is the external face. As above, suppose that $C$ covers a cycle $abc$ in $K_{1,2,2,2}$. If $G$ is an $n$-fold cover of $K_{1,2,2,2}$, then it has $7n$ vertices and $18n$ edges, and hence $11n+2$ faces, by Euler's formula. Since one of the faces has length $3m$ and all others are triangular, a simple calculation gives $n=m+1$. It follows that inside $C$ we have only one of each vertices with labels $a,b$ and $c$, which must then be the vertices of a triangular face $\triangle$ of $G$ by~\eqref{it:shortf}. Moreover, as $n$ is even by \cite{3} we obtain that $m$ is odd and $m \ge 3$. Let $d \notin \{a,b,c\}$ be a vertex of $K_{1,2,2,2}$ connected to all three vertices $a, b$ and $c$. There are $n \ge 4$ vertices labelled $d$ in the annular domain bounded by $C$ and $\triangle$, and so at least one of them is not connected to any vertex of $\triangle$. Any such vertex of $G$ (call it $d'$), is connected to three vertices labelled $a', b'$ and $c'$ lying on $C$. These three vertices cannot form a $2$-path on $C$, as otherwise the one in the middle would have no other edges of $G$ incident to it, by \eqref{it:shortf}. For each of the three domains into which the edges $(d', a'), (d', b')$ and $(d', c')$ split the interior of $C$, consider the intersection of its boundary with $C$. That intersection is a path on $C$, and for at least two of the three domains, that path contains vertices labelled $a,b$ or $c$ of $C$ other than its endpoints. Hence there is at least one domain with this property that does not contain $\triangle$. We now consider all domains constructed in this way (for different choices of $d'$), and out of them, choose a domain $D$ the intersection $P$ of whose boundary with $C$ is the shortest possible. By construction, the path $P \subset C$ contains at least one vertex labelled $a,b$ or $c$ which is not one of its endpoints, which implies that $D$ contains a vertex labelled $d$. As $D$ does not contain $\triangle$, that vertex $d$ is connected to three internal vertices labelled $a,b$ and $c$ of $P$ hence creating a domain that subtends a shorter path of length greater than one; a contradiction. 
\end{proof}

\section{Selecting a particular subgraph of the cover}\label{sec:pptH}

Assume that we have a cover $G$ of $K_{1, 2, 2, 2}$ satisfying the conditions~\eqref{it:minfold}, \eqref{it:max3}, \eqref{it:shortf} and \eqref{it:longf}.
We refer to 3-cycles in $K_{2,2,2}$ as \emph{octahedral} 3-cycles. 
Clearly, there must be at least one octahedral 3-cycle
whose lift is not a union of 3-cycles in $G$ (otherwise, $G$ has an octahedron subgraph and condition~\eqref{it:shortf} means that there are no other vertices). 
Among all long cycles covering an octahedral 3-cycle, 
choose a cycle $C$ which contains no other such long cycle inside the closed domain that it bounds; call this interior domain $F$. 
Without loss of generality, we can assume that $C$ covers  the $3$-cycle $(1,2,3)$. 
%
%
The domain $F$ is not a face of $G$ by condition \eqref{it:longf}.
Furthermore, $C$ contains no chords, and so there are vertices of $G$ lying in $F$, and hence some vertices labelled $0, -1,-2$ and $-3$ lying in $F$. 


In light of the preceding observations, we define $H$ to be a connected component of the lift of the $K_4$ subgraph on the vertices $0,-1,-2,$ and $-3$ lying inside the domain $F$ that contains no other such component in its interior.
Moreover, we define $G'$ to be the semi-cover of $K_{1,2,2,2}$ consisting of $H$ and all the edges and vertices of $G$ that lie in the internal faces of $H$.

\begin{lemma} \label{l:Hfaces}
The semi-cover $G'$ of $K_{1,2,2,2}$ and the subgraph $H\subset G'$ have the following properties.

\begin{enumerate}[label=\emph{(\alph*)},ref=\alph*]
    \item \label{it:s9H}
    $H \subset G'$ is the lift of the subgraph $K_4 \subset K_{1,2,2,2}$ whose vertices are $0, -1, -2, -3$, 
    $H$ is connected, the restriction of the projection map to $H$ is a (genuine) cover of $K_4$, and the outer boundary of $G'$ is a cycle of $H$.

    \item \label{it:s9tri}    
    All $3$-cycles of $G'$ are facial.
    
    \item \label{it:Htri} Any cycle in $G'$ covering an octahedral $3$-cycle is a triangular face. 

    \item \label{it:s9inn3}
    Every component of the lift of an octahedral $3$-cycle in $G'$ that is not a cycle is a path that starts and ends on the boundary of $G'$. Such a path cannot cover $(1,2,3)$ or $(-1,-2,-3)$, so the lift of these triangles in $G'$ consists only of triangles.

    
    \item \label{it:noK4}
    $H$ is not isomorphic to $K_4$. 

    \item \label{it:notouter}
    $G'$ contains at least one triangle labelled $(1,2,3)$. 


    \item \label{it:2conn}
    $H$ is $2$-connected, and so every internal face of $H$ is homeomorphic to a disc. 

    \item \label{it:mult3}
    The cyclic order of vertex labels around any non-triangular face of $H$ is $0,a_1,b_1,0,$ $a_2,b_2, \dots, 0,a_m,b_m$, where $a_i, b_i \in \{-1,-2,-3\}$; in particular, the length of any facial cycle is a multiple of $3$.

    \item \label{it:no6face}
    No internal face of $H$ is hexagonal. 

    \item \label{l:above} 
    If an internal face of $H$ is $3m$-gonal, where $m \ge 1$, and contains $t$ triangles labelled $(1,2,3)$, then $t < \frac23 m$.
\end{enumerate}
\end{lemma}

\begin{proof}

  Property~\eqref{it:s9H} follows from the definition of $H$. Property~\eqref{it:s9tri} follows from~\eqref{it:shortf}.
  Properties~\eqref{it:Htri} and~\eqref{it:s9inn3} follow from the fact that the cycle $C$ contains no long octahedral 3-cycle in its interior.
  $H$ is not isomorphic to $K_4$ (property~\eqref{it:noK4}) because if it were, the middle vertex would be incident only to triangular faces of $H$, and thus no vertex labelled $1,2$ or $3$.
  Property~\eqref{it:notouter} holds because otherwise $H$ would be an outerplanar cubic graph.

  For property~\eqref{it:2conn}, assume that $H$ contains a bridge $(a,b)$. Let $H_1$ be the component of $H$ that contains the vertex $b$ of the bridge when $(a,b)$ is deleted. Let $m$ be the number of vertices $c \ne a,b$ in $H_1$. Then $H_1$ has exactly $m$ vertices $b$ and exactly $m$ vertices $a$, which is a contradiction, as there is a vertex $b$ in $H_1$ not connected to any vertex $a$ in $H_1$. Since $H$ is $3$-regular and contains no bridge it is $2$-connected.

  For property~\eqref{it:mult3}, assume a cycle $C'$ bounds a face $F'$ of $H$. If $C'$ has length $3$, then it bounds a triangular face by~\eqref{it:shortf}. 
  Suppose the length of $C'$ is at least $4$. Then out of any three consecutive vertices along $C'$ we must have at least one $0$ (for if the vertices were $-1,-2,-3$, in this order, then there would be an edge $(-2,0)$ lying outside $F'$ and inside the triangular face $(-1,-2,-3)$). Clearly out of any three consecutive vertices along $C'$ we cannot have more than one $0$, which proves the claim.

  For property~\eqref{it:no6face}, assume that $F'$ is an internal hexagonal face. 
  We say that an internal face of $H$ is \emph{empty} if it contains no vertices labelled $1,2,$ or $3$ inside it, and is \emph{full} otherwise. 
  By assertion~\eqref{it:mult3}, the cyclic order of vertex labels on its boundary $C'$ is $0,a_1,b_1,0,a_2,b_2$, where $a_i, b_i \in \{-1,-2,-3\}$. Moreover, $\{a_1, a_2\} \cap \{b_1, b_2\} = \varnothing$, so without loss of generality we can assume that $a_1=a_2=-1$ and $b_1=-2$. If $b_2 = -2$, then the face $F'$ is empty (there is no $-3$ vertex on $C'$), and so it is a face of $G$, in contradiction with condition \eqref{it:longf}. Suppose $b_2=-3$, so that the sequence of labels along $C'$ is $0,-1,-2,0,-1,-3$. First suppose that $F'$ is full. Then $F'$ contains exactly one triangle $(1,2,3)$, as there is only one $-3$ on $C'$ to connect to $1$. Then the union of the two edges $(1,-2)$ and $(1,-3)$ splits $F'$ into two pentagonal domains, and the two vertices $2$ and $3$ both lie in one of them. But then both vertices $2$ and $3$ are connected to each of the vertices $0,-1$ on the boundary of this pentagonal domain which gives a crossing. Now assume $F'$ is empty (and $b_2$ is still $-3$). We note that the face of $H$ adjacent to the edge $(-1,-2)$ from the opposite side of $F'$ is a triangle $(-1,-2,-3')$ (where we write $-3'$ to indicate that this vertex is different from the vertex $-3$ on $C'$). If the face $F''$ of $H$ adjacent to the first edge $(0,-1)$ in our sequence $0,-1,-2,0,-1,-3$ from the opposite side of $F'$ is a triangle, then $F''$ must be the triangle $(0,-1,-3')$, and so out of the three faces adjacent to the vertex $-1$, two are triangles, and the third is $F'$, which implies that $F'$ must be full. It follows that $F''$ is not a triangle. But then, as $F'$ is empty, it must be a face of $G$, and so replacing two edges $(0,-1)$ of $C'$ by the two chords $(0,-1)$ we increase the number of triangles in $G$, in contradiction with~\eqref{it:max3}. 
  
  For property~\eqref{l:above}, 
 let $F'$ be such a face. Denote $q_i, \, i=1,2,3$, the number of labels $(-i)$ on its boundary $C'$. We have $q_1+q_2+q_3=2m$ and $t \le \min(q_1,q_2,q_3)$. This immediately gives a non-strict inequality $t \le \frac23 m$. Assuming there is an equality, we must necessarily have $m=3s$, for some $s \in \mathbb{N}$, and then $t=q_1=q_2=q_3=2s$. But then every vertex labelled $(-i)$ on $C'$ is connected to vertices labelled $j$ and $k$, both lying in $F'$, where $\{i,j,k\}=\{1,2,3\}$. This means that any cycle in the lift of any triangle $(-i,j,k)$ of $K_{1,2,2,2}$ which starts in $F'$, does not leave (the closure of) $F'$.  So by property~\eqref{it:Htri}, any such lift is the union of triangles. Thus for example, any vertex $-1$ on $C'$ is connected to two vertices labelled $2$ and $3$ in $F'$, which are connected to each other and both connected to the same vertex labelled $1$ lying in $F'$. So for any triangle $\triangle=(1,2,3)$ lying in $F'$, there are vertices $-1,-2,-3$ on $C'$ each of which is connected to the corresponding pair of vertices of $\triangle$. Furthermore, each vertex of $\triangle$ is connected to a vertex labelled $0$ on $C'$; these three $0$ vertices alternate along $C'$ with the above vertices $-1,-2,-3$.
 We denote $T(\triangle)$ the union of all the edges of $G$ connected to the vertices of a triangle $\triangle=(1,2,3)$. Choose one such triangle, $\triangle_1$, construct $T(\triangle_1)$ as above and choose a  path $\gamma_1$ on $C'$ between two $0$ vertices of $T(\triangle_1)$ and not containing the third $0$ vertex of $T(\triangle_1)$. By property~\eqref{it:mult3}, $\gamma_1$ contains at least two vertices of $C'$ with non-zero labels, and so at least one such vertex not belonging to $T(\triangle_1)$. Then there exists a triangle $\triangle_2=(1,2,3)$ lying in $F' \setminus T(\triangle_1)$ such that $T(\triangle_2)$ contains that vertex and that all the six vertices of $C' \cap T(\triangle_2)$ lie on $\gamma_1$. We can then choose a path $\gamma_2 \subset C'$ between two $0$ vertices of $C' \cap T(\triangle_2)$ which is a proper sub-path of $\gamma_1$, and repeat the argument. As the cover is finite, we arrive at a contradiction. 
\end{proof}

For the rest of the paper we consider a planar semi-cover $G'$ of $K_{1,2,2,2}$ and its subgraph $H$ that satisfies the conditions listed in Lemma~\ref{l:Hfaces}, and such that $H$ has the smallest possible number of vertices (i.e. the smallest fold number as a cover of $K_4$).
Our overall aim is to show that such a semi-cover does not exist. In fact the following is a strengthening of Conjecture~\ref{conj: negami}.

\begin{conjecture}\label{conj:apn}
There is no planar semi-cover $G'$ of $K_{1,2, 2, 2}$ having the properties listed in Lemma~\ref{l:Hfaces}. 
\end{conjecture}



       %


\section{Beads and Necklaces}
\label{s:bsn}

Consider the subgraph $H$ of a semi-cover $G'$ satisfying Lemma~\ref{l:Hfaces}. Recall that it is a cover of $K_4$ with vertices $\{0,-1,-2,-3\}$ with all cycles that cover the triangle $(-1,-2,-3)$ being triangular faces of $G'$.

The central argument in the proofs of non-existence of a planar cover for several graphs was based on the study of the `beads' and their arrangements called the `necklaces' \cite{1}. In our setting, a \emph{bead} is a labelled subgraph of $H$ as shown on the left in Figure~\ref{fig:bts}, where $\{i,j,k\}=\{1,2,3\}$.
A maximal subgraph of $H$ consisting of a sequence of beads, every two consecutive ones of which are connected by an edge, together with the two edges attached to the first and to the last bead is called a \emph{string}, as in the middle in Figure~\ref{fig:bts}, where $k \in \{1,2,3\}$ (maximality means that the subgraph is not properly contained in any other such subgraph). Note that all the vertices of a string labelled $-k$ are indeed forced to have the same label (we will say that the string has \emph{type $-k$}), while the labelling of other vertices may be different for different beads. We define a \emph{necklace} as a cyclic string of beads. Our main aim in this section is to prove that $H$ cannot be a necklace.

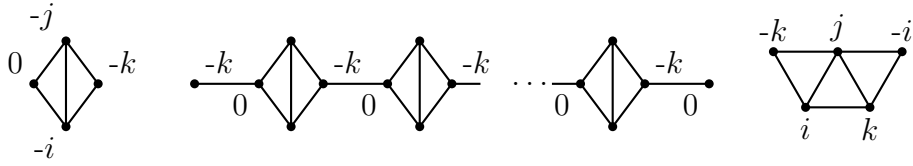
\begin{figure}[ht]
\centering
\begin{tikzpicture}[scale=0.57]
\def \r {6}
\def \rb {1}
\def \rv {3}
\def \rva {4.5}
\def \ger{2/3}
\begin{scope}
  \coordinate (-1) at (0,-\rb); \fill (-1) circle (3pt); \draw (-1) node[below left] {-$i$};
  \coordinate (-2) at (0,\rb); \fill (-2) circle (3pt); \draw (-2) node[above left] {-$j$};
  \coordinate (-3) at (\rv/4,0); \fill (-3) circle (3pt); \draw (-3) node[above right] {-$k$};
  \coordinate (0) at (-\rv/4,0); \fill (0) circle (3pt); \draw (0) node[above left] {0};
  \draw[thick] (-1)--(-3)--(-2)--(-1)--(0)--(-2);
\end{scope}
\begin{scope}[shift={(7*\r/8,0)}]
  \foreach \x in {0,2,5}{
    \coordinate (-1l\x) at (\x*\rv/2,-\rb); \fill (-1l\x) circle (3pt); 
    \coordinate (-2l\x) at (\x*\rv/2,\rb); \fill (-2l\x) circle (3pt); 
    \coordinate (-3l\x) at (\rv/4+\x*\rv/2,0); \fill (-3l\x) circle (3pt); \draw (-3l\x) node[above right] {-$k$};
    \coordinate (0l\x) at (-\rv/4+\x*\rv/2,0); \fill (0l\x) circle (3pt); \draw (0l\x) node[below left] {0};
    \draw[thick] (-1l\x)--(-3l\x)--(-2l\x)--(-1l\x)--(0l\x)--(-2l\x);
};
  \coordinate (0e) at (13*\rv/4,0); \fill (0e) circle (3pt); \draw (0e) node[below left] {0};
  \coordinate (-3e) at (-3*\rv/4,0); \fill (-3e) circle (3pt); \draw (-3e) node[above right] {-$k$};
  \draw[thick] (0l0) -- (-3e) (-3l0) -- (0l2) (-3l2) -- ++(\ger,0) (0l5) -- ++(-\ger,0) (-3l5) -- (0e); \draw (15*\rv/8,0) node {$\dots$};
\end{scope}
\begin{scope}[shift={(3*\r,0)}]
  \coordinate (2l) at (0,\rv/4); \fill (2l) circle (3pt); \draw (2l) node[above] {$j$};
  \coordinate (-1l) at (\rv/2,\rv/4); \fill (-1l) circle (3pt); \draw (-1l) node[above] {-$i$};
  \coordinate (-3l) at (-\rv/2,\rv/4); \fill (-3l) circle (3pt); \draw (-3l) node[above] {-$k$};
  \coordinate (3l) at (\rv/4,{(1-sqrt(3))/4*\rv}); \fill (3l) circle (3pt); \draw (3l) node[below] {$k$};
  \coordinate (1l) at (-\rv/4,{(1-sqrt(3))/4*\rv}); \fill (1l) circle (3pt); \draw (1l) node[below] {$i$};
  \draw[thick] (1l)--(-3l)--(2l)--(-1l)--(3l)--(2l)--(1l)--(3l);
\end{scope}
\end{tikzpicture}
\caption{Left to right: a bead, a string (of type $-k)$ and a trapezium (of type $j$).}
\label{fig:bts}
\end{figure}

We call the vertices $-i$ and $-j$ of the bead on the left in Figure~\ref{fig:bts} the \emph{inner vertices}, and we say that a bead \emph{lies} on the boundary of a face $F$ of $H$ if one of its inner vertices belongs to the boundary of $F$. Clearly, any inner vertex of a bead belongs to a single non-triangular face of $H$ (which can be the external face $F_e$), and it is not hard to see that two inner vertices of a bead must belong to the boundaries of two different non-triangular faces $F'$ and $F''$; in that case, we say that the faces $F'$ and $F''$ \emph{share a bead}. It is also clear (and this is the most valuable property of a bead) that if an inner vertex $-i$ of a bead lies on the boundary of a non-triangular face $F$ of $H$, then both edges $(-i,j)$ and $(-i,k)$ of $G'$ incident to that vertex, where $\{i,j,k\}=\{1,2,3\}$, go in the domain $F$ (that is, lie in the closure of $F$).

Suppose $F$ is an internal, non-triangular face of $H$ having a string $s$ lying on its boundary. A triangle $\triangle$ with the vertices $(1,2,3)$ lying in $F$ has nine edges of $G'$ joining its vertices with some vertices on the boundary of $F$ (because each vertex has degree 5).
We call the these vertices \emph{vertices of attachment \emph{(}of} $\triangle$). We say that a triangle $\triangle$ with the vertices $(1,2,3)$ lying in $F$ is \emph{supported on the string $s$}, if all its vertices of attachment lie on $s$. For a triangle $\triangle$ supported on a string $s$, consider the union $U_\triangle$ of the domains each of which is bounded by two edges between the vertices of $\triangle$ and of $s$, a path on $s$ between the corresponding vertices of attachment (and possibly an edge of $\triangle$). If $U_\triangle$ contains another triangle $\triangle'$  with the vertices $(1,2,3)$, we write $\triangle' \prec \triangle$. Note that the triangle $\triangle'$ is also supported on $s$ and that the relation $\prec$ is a strict partial order on the set of triangles $(1,2,3)$ supported on $s$. We call a triangle $\triangle$ supported on $s$ \emph{minimal}, if it is minimal relative to $\prec$.

In the following, given a string $s$ (of type $-k,\, k \in \{1,2,3\}$) lying on the boundary of a face $F$, we refer to the direction on the path of $s$ lying on the boundary of $F$ from its endpoint labelled $0$ to its other endpoint labelled $-k$ as the direction ``up" (so we think of $s$ as being positioned vertically, as in Figure~\ref{fig:support1}); in this sense, we will speak of a vertex of $s$ lying ``above" another vertex of $s$ or of the ``top" and the ``bottom" vertex of a set of vertices of that path.

We prove the following technical lemma first. 
We call a \emph{trapezium (of type $j$)} 
a labelled subgraph of $H$ shown on the right in Figure~\ref{fig:bts}, where $\{i,j,k\}=\{1,2,3\}$.

\begin{lemma} \label{l:support1}
Suppose $F$ is an internal, non-triangular face of $H$ having a string $s$ of type $-k,\, k \in \{1,2,3\}$, lying on its boundary. 

    \begin{enumerate}[label=\emph{(\alph*)},ref=\alph*]
      \item \label{it:min31} 
        Let $(1,2,3)$ be a minimal triangle supported on the string $s$ and lying in $F$. Then the triangle and the nine edges incident to its vertices are positioned as in one of the three cases given in Figure~\ref{fig:support1}, where $\{i,j,k\}=\{1,2,3\}$ \emph{(}the three cases only differ by the positions of the two edges, $(-k,i)$ and $(0,j)$\emph{)}.

      \item \label{it:from0to-31} 
        Given a non-empty set of triangles $(1,2,3)$ lying in $F$ and supported on the string $s$, consider the set $A$ of all their vertices of attachment. Then the top vertex of $A$ has label $-k$ and the bottom vertex has label $0$.

      \item \label{it:trap-3top1}
        Suppose a trapezium of type $j$ \emph{(}as on the right in Figure~\eqref{fig:bts}\emph{)} lies in the \emph{(}closure of the\emph{)} face $F$, with both vertices $-k$ and $-i$ lying on $s$, where $\{i,j,k\}=\{1,2,3\}$. Then the vertex $-k$ of the trapezium lies above the vertex $-i$ on the string $s$.
    \end{enumerate}
   \end{lemma}

\begin{figure}[ht]
\centering
\begin{tikzpicture}[scale=0.5]
\def \r {6}
\def \rb {1.5}
\def \rv {4} 
\def \rva {4.5}
\def \ger{0.5}
\begin{scope}
  \coordinate (-3t) at (0,7*\rv/4); \fill (-3t) circle (3pt); \draw (-3t) node[below left] {-$k$};
  \coordinate (0t) at (0,5*\rv/4); \fill (0t) circle (3pt); \draw (0t) node[above left] {0};
  \coordinate (-2l) at (-\rb,\rv); \fill (-2l) circle (3pt); \draw (-2l) node[above] {-$j$};
  \coordinate (-1r) at (\rb,\rv); \fill (-1r) circle (3pt); \draw (-1r) node[shift=({-0.15,-0.35})] {-$i$}; %
  \coordinate (-3c) at (0,3*\rv/4); \fill (-3c) circle (3pt); \draw (-3c) node[below left] {-$k$};
  \coordinate (0c) at (0,\rv/4); \fill (0c) circle (3pt); \draw (0c) node[above left] {0};
  \coordinate (-1l) at (-\rb,0); \fill (-1l) circle (3pt); \draw (-1l) node[above] {-$i$};
  \coordinate (-2r) at (\rb,0); \fill (-2r) circle (3pt); \draw (-2r) node[below] {-$j$};
  \coordinate (-3b) at (0,-\rv/4); \fill (-3b) circle (3pt); \draw (-3b) node[below left] {-$k$};
  \coordinate (0b) at (0,-3*\rv/4); \fill (0b) circle (3pt); \draw (0b) node[above left] {0};
  \draw[thick] (-3t)--(0t)--(-2l)--(-3c)--(0c)--(-1l)--(-3b)--(0b) (0t)--(-1r)--(-3c) (-1r)--(-2l) (0c)--(-2r)--(-3b) (-2r)--(-1l);
  \draw[thick] (-3t) -- ++(-\ger,\ger) (-3t) -- ++(\ger,\ger) (0b) -- ++(-\ger,-\ger) (0b) -- ++(\ger,-\ger);
  \coordinate (3) at (5*\rb/4,\rv/2); \fill (3) circle (3pt); \draw (3) node[shift=({0.1,-0.25})] {$k$};
  \coordinate (2) at (9*\rb/4,3*\rv/4); \fill (2) circle (3pt); \draw (2) node[above right] {$j$};
  \coordinate (1) at (9*\rb/4,\rv/4); \fill (1) circle (3pt); \draw (1) node[right] {$i$};
  \draw[thick] (3)--(2)--(1)--(3);
  \draw[ultra thin] (2)--(-1r) -- (3) -- (-2r)--(1) (0c)--(3);
  \draw[ultra thin] (0t) .. controls (2*\rb,\rv) .. (2) (-3t) .. controls (7*\rb/4,3*\rv/2) .. (2) (-3t) .. controls (3*\rb,3*\rv/2) .. (1);
  \draw[ultra thin] (0b) .. controls (7*\rb/4,-\rv/2) .. (1);
  \draw (3/2*\ger,-7*\rv/8) node {$s$}; \draw (2*\rb,-3*\rv/4) node {$F$};
\end{scope}
\begin{scope}[shift={(11*\r/8,0)}]
  \coordinate (-3t) at (0,7*\rv/4); \fill (-3t) circle (3pt); \draw (-3t) node[below left] {-$k$};
  \coordinate (0t) at (0,5*\rv/4); \fill (0t) circle (3pt); \draw (0t) node[above left] {0};
  \coordinate (-2l) at (-\rb,\rv); \fill (-2l) circle (3pt); \draw (-2l) node[above] {-$j$};
  \coordinate (-1r) at (\rb,\rv); \fill (-1r) circle (3pt); \draw (-1r) node[shift=({-0.15,-0.35})] {-$i$}; %
  \coordinate (-3c) at (0,3*\rv/4); \fill (-3c) circle (3pt); \draw (-3c) node[below left] {-$k$};
  \coordinate (0c) at (0,\rv/4); \fill (0c) circle (3pt); \draw (0c) node[above left] {0};
  \coordinate (-1l) at (-\rb,0); \fill (-1l) circle (3pt); \draw (-1l) node[above] {-$i$};
  \coordinate (-2r) at (\rb,0); \fill (-2r) circle (3pt); \draw (-2r) node[below] {-$j$};
  \coordinate (-3b) at (0,-\rv/4); \fill (-3b) circle (3pt); \draw (-3b) node[below left] {-$k$};
  \coordinate (0b) at (0,-3*\rv/4); \fill (0b) circle (3pt); \draw (0b) node[above left] {0};
  \draw[thick] (-3t)--(0t)--(-2l)--(-3c)--(0c)--(-1l)--(-3b)--(0b) (0t)--(-1r)--(-3c) (-1r)--(-2l) (0c)--(-2r)--(-3b) (-2r)--(-1l);
  \draw[thick] (-3t) -- ++(-\ger,\ger) (-3t) -- ++(\ger,\ger) (0b) -- ++(-\ger,-\ger) (0b) -- ++(\ger,-\ger);
  \coordinate (3) at (5*\rb/4,\rv/2); \fill (3) circle (3pt); \draw (3) node[shift=({0.1,-0.25})] {$k$};
  \coordinate (2) at (9*\rb/4,3*\rv/4); \fill (2) circle (3pt); \draw (2) node[above right] {$j$};
  \coordinate (1) at (9*\rb/4,\rv/4); \fill (1) circle (3pt); \draw (1) node[right] {$i$};
  \draw[thick] (3)--(2)--(1)--(3);
  \draw[ultra thin] (2)--(-1r) -- (3) -- (-2r)--(1) (0c)--(3);
  \draw[ultra thin] (0t) .. controls (2*\rb,\rv) .. (2) (-3t) .. controls (7*\rb/4,3*\rv/2) .. (2) (-3b) .. controls (3*\rb/2,-\rv/4) .. (1);
  \draw[ultra thin] (0b) .. controls (7*\rb/4,-\rv/2) .. (1);
  \draw (3/2*\ger,-7*\rv/8) node {$s$}; \draw (2*\rb,-3*\rv/4) node {$F$};
\end{scope}
\begin{scope}[shift={(11*\r/4,0)}]
  \coordinate (-3t) at (0,7*\rv/4); \fill (-3t) circle (3pt); \draw (-3t) node[below left] {-$k$};
  \coordinate (0t) at (0,5*\rv/4); \fill (0t) circle (3pt); \draw (0t) node[above left] {0};
  \coordinate (-2l) at (-\rb,\rv); \fill (-2l) circle (3pt); \draw (-2l) node[above] {-$j$};
  \coordinate (-1r) at (\rb,\rv); \fill (-1r) circle (3pt); \draw (-1r) node[shift=({-0.15,-0.35})] {-$i$}; %
  \coordinate (-3c) at (0,3*\rv/4); \fill (-3c) circle (3pt); \draw (-3c) node[below left] {-$k$};
  \coordinate (0c) at (0,\rv/4); \fill (0c) circle (3pt); \draw (0c) node[above left] {0};
  \coordinate (-1l) at (-\rb,0); \fill (-1l) circle (3pt); \draw (-1l) node[above] {-$i$};
  \coordinate (-2r) at (\rb,0); \fill (-2r) circle (3pt); \draw (-2r) node[below] {-$j$};
  \coordinate (-3b) at (0,-\rv/4); \fill (-3b) circle (3pt); \draw (-3b) node[below left] {-$k$};
  \coordinate (0b) at (0,-3*\rv/4); \fill (0b) circle (3pt); \draw (0b) node[above left] {0};
  \draw[thick] (-3t)--(0t)--(-2l)--(-3c)--(0c)--(-1l)--(-3b)--(0b) (0t)--(-1r)--(-3c) (-1r)--(-2l) (0c)--(-2r)--(-3b) (-2r)--(-1l);
  \draw[thick] (-3t) -- ++(-\ger,\ger) (-3t) -- ++(\ger,\ger) (0b) -- ++(-\ger,-\ger) (0b) -- ++(\ger,-\ger);
  \coordinate (3) at (5*\rb/4,\rv/2); \fill (3) circle (3pt); \draw (3) node[shift=({0.1,-0.25})] {$k$};
  \coordinate (2) at (9*\rb/4,3*\rv/4); \fill (2) circle (3pt); \draw (2) node[above right] {$j$};
  \coordinate (1) at (9*\rb/4,\rv/4); \fill (1) circle (3pt); \draw (1) node[shift=({-0.15,0.3})] {$i$};
  \draw[thick] (3)--(2)--(1)--(3);
  \draw[ultra thin] (2)--(-1r) -- (3) -- (-2r)--(1) (0c)--(3);
  \draw[ultra thin] (0b) .. controls (11*\rb/4,-\rv/2) .. (2) (-3t) .. controls (7*\rb/4,3*\rv/2) .. (2) (-3b) .. controls (3*\rb/2,-\rv/4) .. (1);
  \draw[ultra thin] (0b) .. controls (7*\rb/4,-\rv/2) .. (1);
  \draw (3/2*\ger,-7*\rv/8) node {$s$}; \draw (2*\rb,-3*\rv/4) node {$F$};
\end{scope}
\end{tikzpicture}
\caption{A minimal triangle $(1,2,3)$ supported on the string $s$.}
\label{fig:support1}
\end{figure}
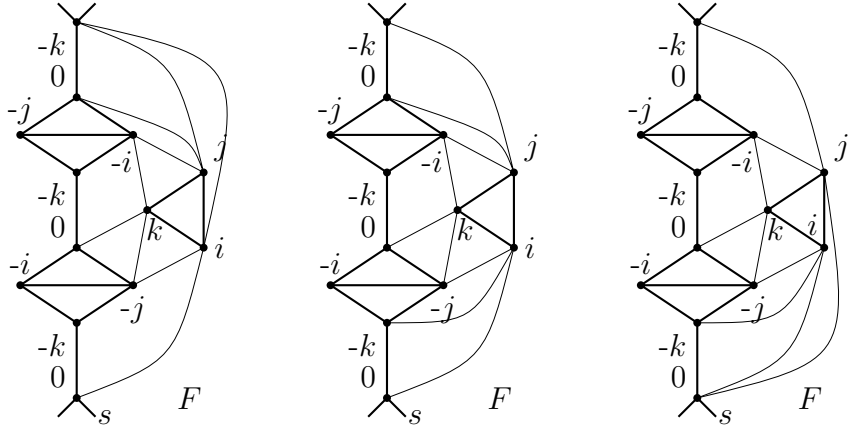

\begin{proof}
   For assertion~\eqref{it:min31} we place the string $s$ vertically and assume (up to reflection in the vertical axis) that the minimal triangle $(1,2,3)$ supported on $s$ lies in the domain $F$ to the right of it, as in Figure~\ref{fig:support1}. All the edges incident to the vertices of the triangle have their endpoints on $s$, and so lie on the path $p$ of $s$ between the bottom and the top vertices of attachment. The vertices $-i$ and $-j$ on $p$ incident to the edges $(k,-i)$ and $(k,-j)$ must be the inner vertices of the beads of $s$, and moreover, the interior of $p$ cannot contain any other inner vertices of the beads of $s$; in particular, these two beads must be consecutive along $s$. Furthermore, both vertices $i$ and $j$ must lie outside the domain bounded by the two edges $(k,-i)$ and $(k,-j)$ and the segment of $p$ between the vertices $-i$ and $-j$ (for otherwise we will not be able to connect both $i$ and $j$ to both $0$ and $-k$ within that domain). Then the path $p$ has to contain at least one vertex $0$ of $s$ lying below both $-i$ and $-j$, and at least one vertex $-k$ of $s$ lying above both $-i$ and $-j$. This forces the position of the edges $(-i,j)$ and $(-j,i)$, and also of the edge $(k,0)$. Then using the fact that $p$ can have no other interior vertices of the beads of $s$ in its interior, it is easy to see that all the other edges except for two are forced. We arrive at one of the three cases shown in Figure~\ref{fig:support1} which only differ by the positions of the edges $(-k,i)$ and $(0,j)$. Note that the proof depends only on the structure of the part of $G'$ lying in (the closure of) $F$ and not on the rest of $G'$.
   
   To prove assertion~\eqref{it:from0to-31} it suffices to consider a set consisting of a single triangle $\triangle$ supported on $s$. Denote $\triangle_{\preceq}$ the set consisting of the triangle $\triangle$ and all the triangles $\triangle' \prec \triangle$. All the triangles in the set $\triangle_{\preceq}$ lie in $F$ and are supported on $s$, and moreover, the top and the bottom attachment points of the set $\triangle_{\preceq}$ are the same as those for the triangle $\triangle$. Let $m$ be the cardinality of the set $\triangle_{\preceq}$. If $m=1$, then $\triangle$ is minimal, and the claim follows from assertion~\eqref{it:min31}. If $m > 1$, choose a minimal triangle $\triangle' \in \triangle_{\preceq}$. Since $\triangle' \prec \triangle$, its top (respectively, bottom) attachment point either coincides or lies below (respectively, above) the top (respectively, the bottom) attachment point of $\triangle$. We now remove from $F$ the minimal triangle $\triangle' \in \triangle_{\preceq}$, together with all the edges connected to its vertices, and replace the part of the string $s$ between the bottom attachment vertex $0$ and the top attachment vertex $-k$ of $\triangle'$ with a single edge $(0,-k)$ (refer to Figure~\ref{fig:support1}). This gives a set $\triangle_{\preceq}$ of cardinality $m-1$, with the same top and bottom points of attachment (note that performing this surgery may violate the property of $G'$ to be a semi-cover, as there may be edges attached to $s$ ``from the other side of $F$"; however, as the proof of assertion~\eqref{it:min31} only depends on the structure of the part of $G'$ lying in $F$ and not on the rest of $G'$, the inductive argument works). 
   
   For assertion~\eqref{it:trap-3top1}, we suppose to the contrary that the vertex $-k$ of the trapezium lies below the vertex $-i$ on the string $s$. Denote $p$ the path of $s$ lying on the boundary of $F$ between these vertices $-i$ and $-k$, and let $p'$ be the one of the two paths $(-i,k,i,-k)$ and $(-i,j,-k)$ lying on the trapezium which lies ``closer" to $s$ (so that the domain $F'$ bounded by $p \cup p'$ contains no vertices of the trapezium inside). First suppose that $p'=(-i,j,-k)$. Then the (interior of the) edge $(j,0')$ lies in $F'$, with $0'$ being an inner vertex of the path $p$. Consider the vertex $v$ of $s$ lying on the boundary of $F$ immediately below the vertex $0'$. It is labelled $-i$ or $-j$, is still an inner vertex of $p$ and is the inner vertex of a bead. Therefore $v$ must be connected to a vertex labelled $k$ lying in the sub-domain of $F'$ bounded by the path $(-k,0')$ of $p$ and the two edges $(j,0')$ and $(j,-k)$. That vertex $k$ belongs to some triangle $(1,2,3)$ lying in $F'$ and supported on $s$. But then the top vertex of attachment of that triangle is either $0'$ or $v \ne -k$, in contradiction with assertion~\eqref{it:from0to-31}. Now suppose that $p'=(-i,k,i,-k)$, and let $v$ be the vertex of $p$ lying immediately above $-k$. Then $v$ is labelled $-i$ or $-j$, is still an inner vertex of $p$ (as $p$ must contain more than one edge) and is the inner vertex of a bead. If $v$ is connected to a vertex of some triangle $(1,2,3)$ other than the triangle belonging to our trapezium, then that triangle is supported on $s$, with the bottom vertex of attachment being $v \ne 0$, which again contradicts assertion~\eqref{it:from0to-31}. So the vertex $v$ has label $-j$ and is connected to the vertices $i$ and $k$ on the path $p'$. But this is again a contradiction, as we cannot connect $i$ to any vertex labelled $0$.
\end{proof}

From Lemma~\ref{l:support1} we obtain the following immediate corollary.

\begin{corollary} \label{c:nonecklace}
   Let $G'$ be a semi-cover satisfying the conditions of Lemma~\ref{l:Hfaces}. 
   \begin{enumerate}[label=\emph{(\alph*)},ref=\alph*]
      \item \label{it:noexts}
        If $G'$ is minimal among all such semi-covers \emph{(}so that $H$ as the cover of $K_4$ has the smallest fold\emph{)} and $s$ is a string shared by an internal face $F$ and the external face $F_e$ of $H$, then no triangle $(1,2,3)$ lying in $F$ can be supported on $s$.

      \item \label{it:noneckl}
        The graph $H$ cannot be a necklace \emph{(}that is, cannot have only one internal, non-triangular face\emph{)}.
   \end{enumerate}
\end{corollary}

\begin{proof}
  Assertion~\eqref{it:noexts} follows from the argument similar to the one in the proof of assertion~\eqref{it:from0to-31} of Lemma~\ref{l:support1}. Assuming that an internal face $F$ of $H$ and the external face $F_e$ share a string and that a triangle $\triangle$ with the vertices $(1,2,3)$ lies in $F$ and is supported on $s$, we can find a minimal triangle $\triangle' \preceq \triangle$ lying in $F$ and supported on $s$. Then removing from $G'$ the triangle $\triangle'$, together with all the edges connected to its vertices, and replacing the part of the string $s$ between the bottom attachment vertex $0$ and the top attachment vertex $-k$ of $\triangle'$ with a single edge $(0,-k)$ (refer to Figure~\ref{fig:support1}) we arrive at a semi-cover which still satisfies conditions conditions of Lemma~\ref{l:Hfaces} and which has a smaller $H$.
  
  Assertion~\eqref{it:noneckl} follows by the same argument if we start with a smallest necklace and note that the removal procedure described in the previous paragraph results in a semi-cover for which $H$ is still a necklace. 
\end{proof}

\section{The graph $H$ with two internal non-triangular faces}\label{s:2f}

In this section, 
we strengthen  Corollary~\ref{c:nonecklace}\eqref{it:noneckl} to the case of the subgraph $H$ having two internal, non-triangular faces.

\begin{proposition} \label{p:twofaces}
  There is no semi-cover $G'$ satisfying the conditions of Lemma~\ref{l:Hfaces} such that $H$ has exactly two internal, non-triangular faces.
\end{proposition}

\begin{proof}
Arguing by contradiction, we assume that such a semi-cover exists and then choose $G'$ to be minimal among all such semi-covers (so that $H$ as the cover of $K_4$ has the smallest fold).

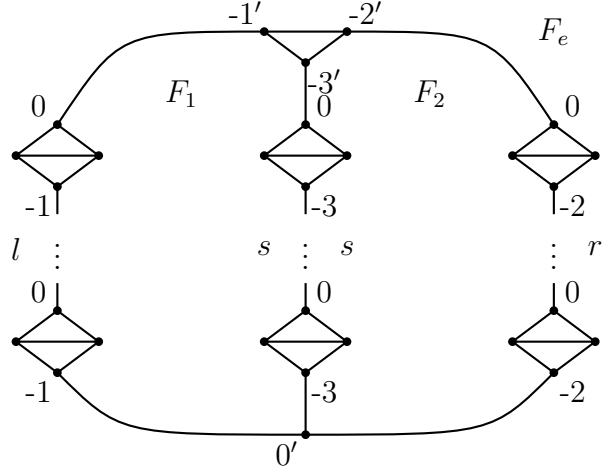
\begin{figure}[ht]
\centering
\begin{tikzpicture}[scale=0.55]
\def \r {6}
\def \rb {1}
\def \rv {3}
\def \rva {4.5}
\def \ger{2/3}
  \coordinate (0b) at (0,-\rv/4); \fill (0b) circle (3pt); \draw (0b) node[shift=({-0.25,-0.25})] {$0'$};
  \coordinate (-3t) at (0,5*\rv/4+\rva); \fill (-3t) circle (3pt); \draw (-3t) node[shift=({0.25,-0.25})] {-$3'$};
  \coordinate (-1t) at (-\rb,3*\rv/2+\rva); \fill (-1t) circle (3pt); \draw (-1t) node[shift=({-0.25,0.25})] {-$1'$};
  \coordinate (-2t) at (\rb,3*\rv/2+\rva); \fill (-2t) circle (3pt); \draw (-2t) node[shift=({0.25,0.25})] {-$2'$};
  \draw[thick] (-3t)--(-2t)--(-1t)--(-3t);
  \foreach \x in {0,1}{
    \coordinate (-1l\x) at (-\r,\rv/4+\x*\rva); \fill (-1l\x) circle (3pt); \draw (-1l\x) node[shift=({-0.25,-0.25})] {-1};
    \coordinate (-2l\x) at (-\r-\rb,\rv/2+\x*\rva); \fill (-2l\x) circle (3pt); 
    \coordinate (-3l\x) at (-\r+\rb,\rv/2+\x*\rva); \fill (-3l\x) circle (3pt); 
    \coordinate (0l\x) at (-\r,3*\rv/4+\x*\rva); \fill (0l\x) circle (3pt); \draw (0l\x) node[shift=({-0.25,0.25})] {0};
    \draw[thick] (-2l\x)--(-1l\x)--(-3l\x)--(0l\x)--(-2l\x)--(-3l\x);
};
  \draw[thick] (0l0) -- ++(0,\ger) (-1l1) -- ++(0,-\ger); \draw (-\r,\rv/2+\rva/2) node {$\vdots$};
  \foreach \x in {0,1}{
    \coordinate (-2r\x) at (\r,\rv/4+\x*\rva); \fill (-2r\x) circle (3pt); \draw (-2r\x) node[shift=({0.25,-0.25})] {-2};
    \coordinate (-1r\x) at (\r+\rb,\rv/2+\x*\rva); \fill (-1r\x) circle (3pt); 
    \coordinate (-3r\x) at (\r-\rb,\rv/2+\x*\rva); \fill (-3r\x) circle (3pt); 
    \coordinate (0r\x) at (\r,3*\rv/4+\x*\rva); \fill (0r\x) circle (3pt); \draw (0r\x) node[shift=({0.25,0.25})] {0};
    \draw[thick] (-1r\x)--(-2r\x)--(-3r\x)--(0r\x)--(-1r\x)--(-3r\x);
};
  \draw[thick] (0r0) -- ++(0,\ger) (-2r1) -- ++(0,-\ger); \draw (\r,\rv/2+\rva/2) node {$\vdots$};
  \foreach \x in {0,1}{
    \coordinate (-3c\x) at (0,\rv/4+\x*\rva); \fill (-3c\x) circle (3pt); \draw (-3c\x) node[shift=({0.25,-0.25})] {-3};
    \coordinate (-1c\x) at (\rb,\rv/2+\x*\rva); \fill (-1c\x) circle (3pt); 
    \coordinate (-2c\x) at (-\rb,\rv/2+\x*\rva); \fill (-2c\x) circle (3pt); 
    \coordinate (0c\x) at (0,3*\rv/4+\x*\rva); \fill (0c\x) circle (3pt); \draw (0c\x) node[shift=({0.25,0.25})] {0};
    \draw[thick] (-1c\x)--(-3c\x)--(-2c\x)--(0c\x)--(-1c\x)--(-2c\x);
};
  \draw[thick] (0c0) -- ++(0,\ger) (-3c1) -- ++(0,-\ger); \draw (0,\rv/2+\rva/2) node {$\vdots$};
  \draw[thick] (0b) -- (-3c0) (0c1) -- (-3t);
  \draw[thick] (0b) .. controls (-3*\r/4,-\rv/4) .. (-1l0) (0b) .. controls (3*\r/4,-\rv/4) .. (-2r0);
  \draw[thick] (-1t) .. controls (-3*\r/4,3*\rv/2+\rva) .. (0l1) (-2t) .. controls (3*\r/4,3*\rv/2+\rva) .. (0r1);
  \draw (-\r/2,\rv+\rva) node {$F_1$}; \draw (\r/2,\rv+\rva) node {$F_2$}; \draw (\r,3*\rv/2+\rva) node {$F_e$};
  \draw (-\r-\rb,\rv/2+\rva/2) node {$l$}; \draw (-\rb,\rv/2+\rva/2) node {$s$};
  \draw (\rb,\rv/2+\rva/2) node {$s$}; \draw (\r+\rb,\rv/2+\rva/2) node {$r$};
\end{tikzpicture}
\caption{$H$ has two internal, non-triangular faces $F_1$ and $F_2$.}
\label{fig:H2}
\end{figure}

Our starting point, the drawing of $H$ (with some labels) is as the one shown in Figure~\ref{fig:H2}. It consists of three strings joining the vertex $0'$ (at the bottom in Figure~\ref{fig:H2}) with the vertices of the triangle $(-1',-2',-3')$ (on the top). We call $l$ (respectively, $r$) the string shared by $F_1$ (respectively, $F_2$) and $F_e$, and we call $s$ the sting shared by $F_1$ and $F_2$. 

We first note that no triangle $(1,2,3)$ lying in the face $F_1$ (respectively, in the face $F_2$) can be supported on the string $l$ (respectively, on the string $r$). This follows from Corollary~\ref{c:nonecklace}\eqref{it:noexts} (more precisely, from its proof, as the removal procedure described there results in a smaller semi-cover still having two internal, non-triangular faces).

Moreover, regardless of the labelling of individual beads in Figure~\ref{fig:H2}, the face $F_1$ lies in the inner domains bounded by a cycle of $H$ with vertices $0,-1,-3$  and by a cycle of $H$ with vertices $0,-1,-2$. By Lemma~\ref{l:Hfaces}\eqref{it:s9inn3}, all the paths $(-3,1,2)$ and $(-2,1,3)$ in the closure of $F_1$ lie on the corresponding triangles. As by Lemma~\ref{l:Hfaces}\eqref{it:s9tri}, the cycles of $G'$ covering the cycle $(1,2,3)$ are also triangles, we obtain that every triangle $(1,2,3)$ lying in $F_1$ is a subgraph of a trapezium of type $1$ (see Figure~\ref{fig:bts}), with both vertices $-2$ and $-3$ lying on the boundary of $F_1$. A similar argument shows that every triangle $(1,2,3)$ lying in $F_2$ is a subgraph of a trapezium of type $2$, with the vertices $-1$ and $-3$ lying on the boundary of $F_2$.

Consider the vertex $-3$ of $s$ lying immediately above the vertex $0'$ (in the notation of Figure~\ref{fig:H2}). Without loss of generality we can assume that it is a vertex of a trapezium $T$ of type $2$ lying in (the closure of) the face $F_2$, with the vertices $-3$ and $-1$ lying on the boundary of $F_2$. From Lemma~\ref{l:support1}\eqref{it:trap-3top1} it follows that the vertex $-1$ of $T$ lies on the string $r$. If there is at least one inner vertex of a bead in the interior of the path of $r$ lying on the boundary of $F_2$ and joining $0'$ to the vertex $-1$ of $T$, then that vertex is connected to vertices labelled $2$ and $i \in \{1,3\}$, of which at least one does not belong to $T$, and hence lies on a triangle $(1,2,3)$ which then must be supported on $r$, a contradiction. It follows that the trapezium $T$ is positioned as shown in Figure~\ref{fig:trapat-3}, with the edge $(2,0')$ forced (clearly the path $(-3,1,3,-1)$ of the trapezium cannot lie below the path $(-3,2,-1)$ for we need to connect each of the vertices $1, 3$ to each of the vertices $0', -2$).

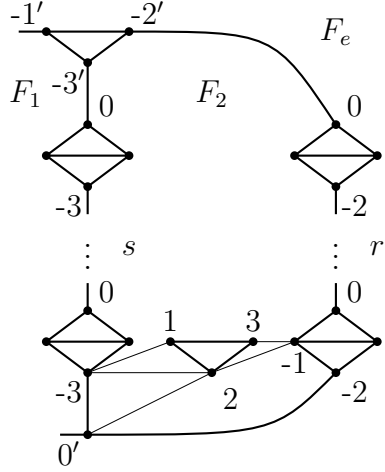
\begin{figure}[ht]
\centering
\begin{tikzpicture}[scale=0.55]
\def \r {6}
\def \rb {1}
\def \rv {3}
\def \rva {4.5}
\def \ger{2/3}
  \coordinate (0b) at (0,-\rv/4); \fill (0b) circle (3pt); \draw (0b) node[shift=({-0.25,-0.25})] {$0'$};
  \coordinate (-3t) at (0,5*\rv/4+\rva); \fill (-3t) circle (3pt); \draw (-3t) node[shift=({-0.25,-0.25})] {-$3'$};
  \coordinate (-1t) at (-\rb,3*\rv/2+\rva); \fill (-1t) circle (3pt); \draw (-1t) node[shift=({-0.25,0.25})] {-$1'$};
  \coordinate (-2t) at (\rb,3*\rv/2+\rva); \fill (-2t) circle (3pt); \draw (-2t) node[shift=({0.25,0.25})] {-$2'$};
  \draw[thick] (-3t)--(-2t)--(-1t)--(-3t);
  \draw[thick] (0b) -- ++(-\ger,0) (-1t) -- ++(-\ger,0);

  \foreach \x in {0,1}{
    \coordinate (-2r\x) at (\r,\rv/4+\x*\rva); \fill (-2r\x) circle (3pt); \draw (-2r\x) node[shift=({0.25,-0.25})] {-2};
    \coordinate (-1r\x) at (\r+\rb,\rv/2+\x*\rva); \fill (-1r\x) circle (3pt); 
    \coordinate (-3r\x) at (\r-\rb,\rv/2+\x*\rva); \fill (-3r\x) circle (3pt); 
    \coordinate (0r\x) at (\r,3*\rv/4+\x*\rva); \fill (0r\x) circle (3pt); \draw (0r\x) node[shift=({0.25,0.25})] {0};
    \draw[thick] (-1r\x)--(-2r\x)--(-3r\x)--(0r\x)--(-1r\x)--(-3r\x);
};
  \draw[thick] (0r0) -- ++(0,\ger) (-2r1) -- ++(0,-\ger); \draw (\r,\rv/2+\rva/2) node {$\vdots$};
  \draw (-3r0) node[below] {-1};
  \foreach \x in {0,1}{
    \coordinate (-3c\x) at (0,\rv/4+\x*\rva); \fill (-3c\x) circle (3pt); \draw (-3c\x) node[shift=({-0.25,-0.25})] {-3};
    \coordinate (-1c\x) at (\rb,\rv/2+\x*\rva); \fill (-1c\x) circle (3pt); 
    \coordinate (-2c\x) at (-\rb,\rv/2+\x*\rva); \fill (-2c\x) circle (3pt); 
    \coordinate (0c\x) at (0,3*\rv/4+\x*\rva); \fill (0c\x) circle (3pt); \draw (0c\x) node[shift=({0.25,0.25})] {0};
    \draw[thick] (-1c\x)--(-3c\x)--(-2c\x)--(0c\x)--(-1c\x)--(-2c\x);
};
  \draw[thick] (0c0) -- ++(0,\ger) (-3c1) -- ++(0,-\ger); \draw (0,\rv/2+\rva/2) node {$\vdots$};
  \draw[thick] (0b) -- (-3c0) (0c1) -- (-3t);
  \draw[thick] (0b) .. controls (3*\r/4,-\rv/4) .. (-2r0);
  \draw[thick] (-2t) .. controls (3*\r/4,3*\rv/2+\rva) .. (0r1);
  \draw (-\r/4,\rv+\rva) node {$F_1$}; \draw (\r/2,\rv+\rva) node {$F_2$}; \draw (\r,3*\rv/2+\rva) node {$F_e$};
  \draw (\rb,\rv/2+\rva/2) node {$s$}; \draw (\r+\rb,\rv/2+\rva/2) node {$r$};
    \coordinate (1) at (\r/2-\rb,\rv/2); \fill (1) circle (3pt); \draw (1) node[above] {1};
    \coordinate (3) at (\r/2+\rb,\rv/2); \fill (3) circle (3pt); \draw (3) node[above] {3};
    \coordinate (2) at (\r/2,\rv/4); \fill (2) circle (3pt); \draw (2) node[shift=({0.25,-0.35})] {2};
  \draw[thick] (1) -- (2) -- (3) -- (1);
  \draw[ultra thin] (1)--(-3c0)--(2)--(-3r0)--(3) (2)--(0b);
\end{tikzpicture}
\caption{The trapezium lying in $F_2$.}
\label{fig:trapat-3}
\end{figure}

Consider the edge $(1,0)$. First suppose that its endpoint $0$ lies on the string $r$. This cannot be the bottom vertex $0$ on $r$ lying above the vertex $-1$ of $T$, as we need an edge $(3,-2)$. If this is any other vertex labelled $0$ of $r$, then the path of $r$ between it and the vertex $-1$ of $T$ contains at least one inner vertex of a bead. That vertex must be connected to a vertex labelled $2$ which then is one of the vertices of a triangle $(1,2,3)$ supported on $r$, which is again a contradiction. Thus the endpoint $0$ of the edge $(1,0)$ lies on $c$. Let $v$ be the vertex of $c$ on the boundary of $F_2$ lying immediately below that vertex $0$. Then $v$ is connected to a vertex labelled $3$ which is not a vertex of $T$. Then this vertex $3$ lies on a triangle $(1,2,3)$ which is supported on $s$ and whose top vertex of attachment to $s$ is either $0$ or $v \ne -3$ which contradicts Lemma~\ref{l:support1}\eqref{it:from0to-31}. 
\end{proof}


Next we establish the following important property of the subgraph $H$. 

\begin{proposition} \label{p:share}
  There is no semi-cover $G'$ satisfying the conditions of Lemma~\ref{l:Hfaces}
  in which two internal faces of $H$ of length at most $3m$ each \emph{(}$m \ge 3$\emph{)} share $m-2$ beads.
\end{proposition}


For example, when $m=3$, we obtain that two $9$-faces of $H$ cannot share a bead. Note that a face of length $3l, \, l \ge 3$, cannot have more than $l$ beads on its boundary, and if it has exactly $l$, we obtain a necklace which is not possible by Corollary~\ref{c:nonecklace}\eqref{it:noneckl}. So the proposition effectively applies only when both faces are of length $3m$, or when one is of length $3m$, and another one, of length $3(m-1)$.

During the proof we will use the following easy fact which, in the setting of Proposition~\ref{p:share}, will strengthen Lemma~\ref{l:Hfaces}\eqref{it:s9inn3}.

\begin{remark} \label{rem:maxpath} 
%
%
  For $\{i,j,k\}=\{1,2,3\}$, consider the components of the lift of  $(-i,j,k)$ that are paths in $G'$. By Lemma~\ref{l:Hfaces}\eqref{it:s9inn3} these paths start and end on the boundary of $G'$, so their endpoints are distinct vertices with label $-i$. In particular, if the boundary of $G'$ has one or zero vertices labelled $-i$, there are no such paths, and the lift consists of only triangular faces.
\end{remark}

\begin{proof}[Proof of Proposition \ref{p:share}]
  Suppose to the contrary that two faces $F_1$ and $F_2$ of $H$ of lengths $3l_1$ and $3l_2$ respectively, where $l_1, l_2 \le m$, share $m-2$ beads. We can assume that $l_1,l_2 \in \{m-1,m\}$. First prove that all the beads shared by $F_1$ and $F_2$ lie on a single string. The argument is easier seen if we consider the quotient $H'$ of the graph $H$ obtained by contracting every triangle $(-1,-2,-3)$ to a single vertex. The resulting graph is again planar, and is cubic and bipartite, with beads projecting to double edges. For $i=1,2$, the face $F_i$ of $H$ projects to a face $F_i'$ of $H'$ of length $2l_i$. Then the faces $F_1'$ and $F_2'$ share $m-2$ double edges. Note that two pairs of double edges are vertex-disjoint (for there are no triple edges and the graph $H'$ is cubic), so on the boundary of $F'_i, \, i=1,2$, there are $2l_i-2(m-2) \in \{2,4\}$ vertices not incident to any double edges shared by $F_1'$ and $F_2'$. Each of these vertices has a single edge incident to it and going to the outside of $F_i'$. Moreover, no such edge may border the other face $F'_{3-i}$ on both of its sides (as $H$ and hence $H'$ are at least $2$-edge connected). Now if $l_i=m-1$, we have exactly two such edges, and then they must be incident to the two consecutive vertices of $F_i'$ which implies that $F_1'$ and $F_2'$ share a path of length $2l_i-1=2m-3$ having $m-2$ pairs of double edges (which lifts to a string of $m-2$ beads shared by $F_1$ and $F_2$ in $H$). If $l_1=l_2=m$, then we have four such edges incident to four vertices on the boundary of $F_1'$. If these vertices are consecutive, we arrive at the same conclusion as above, and in all the other cases, the face $F_2'$ appears to be longer than $2m$.

  We first suppose that both faces are of length $3m$. Then our starting configuration, with some labelling, is as in Figure~\ref{fig:sharem-2}, with $\{u_1, u_2\} = \{-1,-2\}$ and with $\{v_1, v_2\} \subset \{-1,-2,-3\}$ such that $v_1 \ne -1, u_1$ and $v_2 \ne -2, u_2$.

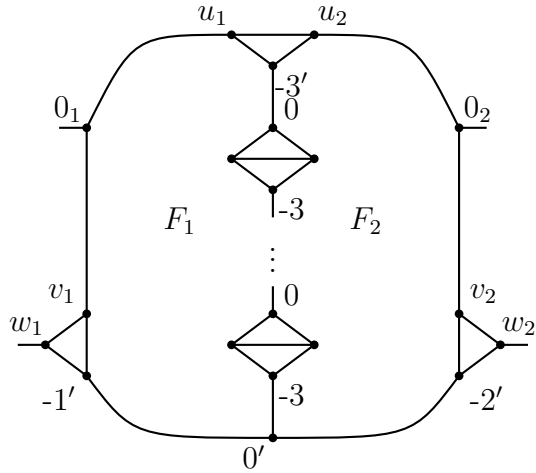
\begin{figure}[ht]
\centering
\begin{tikzpicture}[scale=0.55]
\def \r {4.5}
\def \rb {1}
\def \rv {3}
\def \rva {4.5}
\def \ger{2/3}
  \coordinate (0b) at (0,-\rv/4); \fill (0b) circle (3pt); \draw (0b) node[shift=({-0.25,-0.25})] {$0'$};
  \coordinate (-3t) at (0,5*\rv/4+\rva); \fill (-3t) circle (3pt); \draw (-3t) node[shift=({0.25,-0.25})] {-$3'$};
  \coordinate (-1t) at (-\rb,3*\rv/2+\rva); \fill (-1t) circle (3pt); \draw (-1t) node[shift=({-0.25,0.25})] {$u_1$}; 
  \coordinate (-2t) at (\rb,3*\rv/2+\rva); \fill (-2t) circle (3pt); \draw (-2t) node[shift=({0.25,0.25})] {$u_2$};
  \draw[thick] (-3t)--(-2t)--(-1t)--(-3t);
  \coordinate (-1l) at (-\r,\rv/4); \fill (-1l) circle (3pt); \draw (-1l) node[below left] {-$1'$}; 
  \coordinate (-2l) at (-\r-\rb,\rv/2); \fill (-2l) circle (3pt); \draw (-2l) node[shift=({-0.25,0.25})] {$w_1$};
  \coordinate (-3l) at (-\r,3*\rv/4); \fill (-3l) circle (3pt); \draw (-3l) node[above left] {$v_1$};
  \coordinate (0l) at (-\r,3*\rv/4+\rva); \fill (0l) circle (3pt); \draw (0l) node[shift=({-0.25,0.25})] {$0_1$};
  \draw[thick] (-3l)--(-2l)--(-1l)--(-3l)--(0l);
  \draw[thick] (0l) -- ++(-\ger,0) (-2l) -- ++(-\ger,0);
  \coordinate (-2r) at (\r,\rv/4); \fill (-2r) circle (3pt); \draw (-2r) node[below right] {-$2'$}; 
  \coordinate (-1r) at (\r+\rb,\rv/2); \fill (-1r) circle (3pt); \draw (-1r) node[shift=({0.25,0.25})] {$w_2$};
  \coordinate (-3r) at (\r,3*\rv/4); \fill (-3r) circle (3pt); \draw (-3r) node[above right] {$v_2$};
  \coordinate (0r) at (\r,3*\rv/4+\rva); \fill (0r) circle (3pt); \draw (0r) node[shift=({0.25,0.25})] {$0_2$};
  \draw[thick] (-3r)--(-2r)--(-1r)--(-3r)--(0r);
  \draw[thick] (0r) -- ++(\ger,0) (-1r) -- ++(\ger,0);
  \foreach \x in {0,1}{
    \coordinate (-3c\x) at (0,\rv/4+\x*\rva); \fill (-3c\x) circle (3pt); \draw (-3c\x) node[shift=({0.25,-0.25})] {-3};
    \coordinate (-1c\x) at (\rb,\rv/2+\x*\rva); \fill (-1c\x) circle (3pt); 
    \coordinate (-2c\x) at (-\rb,\rv/2+\x*\rva); \fill (-2c\x) circle (3pt); 
    \coordinate (0c\x) at (0,3*\rv/4+\x*\rva); \fill (0c\x) circle (3pt); \draw (0c\x) node[shift=({0.25,0.25})] {0};
    \draw[thick] (-1c\x)--(-3c\x)--(-2c\x)--(0c\x)--(-1c\x)--(-2c\x);
};
  \draw[thick] (0c0) -- ++(0,\ger) (-3c1) -- ++(0,-\ger); \draw (0,\rv/2+\rva/2) node {$\vdots$};
  \draw[thick] (0b) -- (-3c0) (0c1) -- (-3t);
  \draw[thick] (0b) .. controls (-3*\r/4,-\rv/4) .. (-1l) (0b) .. controls (3*\r/4,-\rv/4) .. (-2r);
  \draw[thick] (-1t) .. controls (-3*\r/4,3*\rv/2+\rva) .. (0l) (-2t) .. controls (3*\r/4,3*\rv/2+\rva) .. (0r);
  \draw (-\r/2,\rva) node {$F_1$}; \draw (\r/2,\rva) node {$F_2$}; 
\end{tikzpicture}
\caption{Two $3m$-faces $F_1$ and $F_2$ of $H$ which share a string of $m-2$ beads.}
\label{fig:sharem-2}
\end{figure}

  Let $a_i$ be the number triangles $(1,2,3)$ lying in the face $F_i, \, i=1,2$, and let $b_1$ (respectively, $b_2$) be the number of beads lying on the string which have vertex $-1$ (respectively, $-2$) on the boundary of $F_1$. Clearly, $b_1+b_2=m-2$. There are $m-1$ vertices labelled $-3$ on the string, which gives $a_1+a_2 \ge m-1$.

  We consider two cases.

  First suppose that $u_1=-2,\, u_2=-1$. Then $v_1=v_2=-3$ and each of the faces $F_1, F_2$ has a single vertex labelled $-1$ and a single vertex labelled $-2$ on its boundary which is not an inner vertex of a bead. By Remark~\ref{rem:maxpath} we obtain that all the paths $(-1,2,3)$ and $(-2,1,3)$ lying in the closure of $F_1 \cup F_2$ lie on the corresponding triangles. In particular, for each of the vertices $-1', u_1$ (respectively, $-2', u_2$), the pair of edges which connect them to the vertices of the triangles $(1,2,3)$ lies either entirely inside or entirely outside the domain $F_1$ (respectively,~$F_2$).

  Suppose that both pairs of edges $(-1',2), (-1',3)$ and $(u_1,1), (u_1,3)$ go in $F_1$. We can then construct a necklace as follows: we retain the string, the part of $G'$ lying in $F_1$, the triangles $(-3',u_1,u_2)$ and $(-1',v_1,w_1)$ (recall that $u_1=w_1=-2,\, u_2=-1,\, v_1=-3$), draw the new edges $(0',w_1)$ and $(0_1,u_2)$, and remove everything else from $G'$. This gives the drawing as on the left in Figure~\ref{fig:bothin} (with $u_1=-2,\, u_2=-1$) leading to a contradiction with Corollary~\ref{c:nonecklace}\eqref{it:noneckl}. Obviously, a similar argument applies when both pairs of edges $(-2',1), (-2',3)$ and $(u_2,2), (u_2,3)$ go in $F_2$.

  Next suppose that both pairs of edges $(-1',2), (-1',3)$ and $(u_1,1), (u_1,3)$ go out of $F_1$. Counting the edges $(-1,3)$ and $(-2,3)$ in $F_1$ we obtain $a_1=b_1=b_2$, and so $m=b_1+b_2=2a_1$. Counting the edges $(-3,1)$ incident to the vertices $-3$ lying on the string we obtain $a_1+a_2 \ge m+1$ which gives $a_2 \ge a_1+1$. Then counting the edges $(-1,3)$ and $(-2,3)$ in $F_2$ we obtain that $a_2=a_1+1$ and that both pairs of edges $(-2',1), (-2',3)$ and $(u_2,2), (u_2,3)$ must go in $F_2$ which reduces this case to the one above.

Now suppose that the edges $(u_1,1),
(u_1,3)$ go in $F_1$, and the edges $(-1',2)$, $(-1',3)$ go out. The count similar to the above gives $a_1=b_1=b_2+1, \; m= 2a_1-1$. There are $m+1=2a_1$ vertices labelled $-3$ on the string; counting the edges $(-3,1)$ incident to them we get $a_2 \ge a_1$. Moreover, there are $b_2+1=a_1$ vertices labelled $-1$ on the boundary of $F_2$ which gives $a_1 \ge a_2$. We deduce that $a_1=a_2$, and hence all the edges $(-3,1)$ and $(-3,2)$ lying in $F_1 \cup F_2$ are incident to the vertices $-3$ on the string. Then both edges $(v_1,1)$ and $(v_1,2)$ go out of $F_1$ (recall that $v_1=-3$). Furthermore, we have $m=2a_1$ vertices labelled $0$ on the string, not counting the vertex $0'$. As there are $2a_1$ triangles $(1,2,3)$ lying in $F_1 \cup F_2$ in total, no two out of the three vertices $0', 0_1$ and $0_2$ can be connected to vertices lying in $F_1 \cup F_2$ and having the same label $i \in \{1,2,3\}$. In particular, the set of labels $i \in \{1,2,3\}$ such that the edge $(0',i)$ lies in $F_1$ and the set of labels $j \in \{1,2,3\}$ such that the edge $(0_1,j)$ lies in $F_1$ are disjoint. We can now construct a necklace as follows: we retain the string, the part of $G'$ lying in $F_1$, the triangles $(-3',u_1,u_2)$, remove the the triangle $(-1',v_1,w_1)$, identify $0_1$ with $0'$ and draw the new edge $(0_1,u_2)$, and then remove everything else from $G'$. The surgery is showing in the middle and on the right in Figure~\ref{fig:bothin} (in the drawing, in the middle, each vertex $0', 0_1$ has one edge lying in $F_1$, $(0',i)$ and $(0_1,j)$ respectively; there can be more of such edges or there can be none, but in any case, all the endpoints of these edges have different labels). We arrive at a contradiction with Corollary~\ref{c:nonecklace}\eqref{it:noneckl}.

\begin{figure}[ht]
\centering
\begin{tikzpicture}[scale=0.5]
\def \r {4.5}
\def \rb {1}
\def \rv {3}
\def \rva {4.5}
\def \ger{1} 
\begin{scope}
  \coordinate (0b) at (0,-\rv/4); \fill (0b) circle (3pt); \draw (0b) node[shift=({-0.25,-0.25})] {$0'$};
  \coordinate (-3t) at (0,5*\rv/4+\rva); \fill (-3t) circle (3pt); \draw (-3t) node[shift=({0.25,-0.25})] {-$3'$};
  \coordinate (-1t) at (-\rb,3*\rv/2+\rva); \fill (-1t) circle (3pt); \draw (-1t) node[shift=({-0.4,-0.25})] {$u_1$}; 
  \coordinate (-2t) at (\rb,3*\rv/2+\rva); \fill (-2t) circle (3pt); \draw (-2t) node[shift=({0,-0.35})] {$u_2$};
  \draw[thick] (-3t)--(-2t)--(-1t)--(-3t);
  \draw[ultra thin] (-1t) -- ++(\ger/3,-{sqrt(8)/3*\ger}) (-1t) -- ++(-\ger/3,-{sqrt(8)/3*\ger});
  \coordinate (-1l) at (-\r,\rv/4); \fill (-1l) circle (3pt); \draw (-1l) node[shift=({0.35,0})] {-$1'$}; 
  \coordinate (-2l) at (-\r-\rb,\rv/2); \fill (-2l) circle (3pt); \draw (-2l) node[shift=({0,0.35})] {$-2$};
  \coordinate (-3l) at (-\r,3*\rv/4); \fill (-3l) circle (3pt); \draw (-3l) node[right] {-$3$};
  \coordinate (0l) at (-\r,3*\rv/4+\rva); \fill (0l) circle (3pt); \draw (0l) node[shift=({-0.25,0.25})] {$0_1$};
  \draw[thick] (-3l)--(-2l)--(-1l)--(-3l)--(0l);
  \draw[ultra thin] (-1l) -- ++({sqrt(3)/2*\ger},\ger/2) (-1l) -- ++(\ger/2,{sqrt(3)/2*\ger});
  \foreach \x in {0,1}{
    \coordinate (-3c\x) at (0,\rv/4+\x*\rva); \fill (-3c\x) circle (3pt); \draw (-3c\x) node[shift=({0.25,-0.25})] {-3};
    \coordinate (-1c\x) at (\rb,\rv/2+\x*\rva); \fill (-1c\x) circle (3pt); 
    \coordinate (-2c\x) at (-\rb,\rv/2+\x*\rva); \fill (-2c\x) circle (3pt); 
    \coordinate (0c\x) at (0,3*\rv/4+\x*\rva); \fill (0c\x) circle (3pt); \draw (0c\x) node[shift=({0.25,0.25})] {0};
    \draw[thick] (-1c\x)--(-3c\x)--(-2c\x)--(0c\x)--(-1c\x)--(-2c\x);
};
  \draw[thick] (0c0) -- ++(0,\ger) (-3c1) -- ++(0,-\ger); \draw (0,\rv/2+\rva/2) node {$\vdots$};
  \draw[thick] (0b) -- (-3c0) (0c1) -- (-3t);
  \draw[thick] (0b) .. controls (-3*\r/4,-\rv/4) .. (-1l);
  \draw[thick] (-1t) .. controls (-3*\r/4,3*\rv/2+\rva) .. (0l);
  \draw[thick, dashed](-2t) .. controls (-\r,7*\rv/4+\rva) .. (0l) (0b) .. controls (-\r,-3*\rv/8) .. (-2l);
  \draw (-\r/2,\rva) node {$F_1$}; 
\end{scope}
\begin{scope}[shift={(1.8*\r,0)}]
  \coordinate (0b) at (0,-\rv/4); \fill (0b) circle (3pt); \draw (0b) node[shift=({-0.25,-0.25})] {$0'$};
  \coordinate (-3t) at (0,5*\rv/4+\rva); \fill (-3t) circle (3pt); \draw (-3t) node[shift=({0.25,-0.25})] {-$3'$};
  \coordinate (-1t) at (-\rb,3*\rv/2+\rva); \fill (-1t) circle (3pt); \draw (-1t) node[shift=({-0.4,-0.25})] {-$2$}; %
  \coordinate (-2t) at (\rb,3*\rv/2+\rva); \fill (-2t) circle (3pt); \draw (-2t) node[shift=({0,-0.35})] {-$1$};%
  \draw[thick] (-3t)--(-2t)--(-1t)--(-3t);
  \draw[ultra thin] (-1t) -- ++(\ger/3,-{sqrt(8)/3*\ger}) (-1t) -- ++(-\ger/3,-{sqrt(8)/3*\ger});
  \coordinate (-1l) at (-\r,\rv/4); \fill (-1l) circle (3pt); \draw (-1l) node[shift=({0.35,0})] {-$1'$}; 
  \coordinate (-2l) at (-\r-\rb,\rv/2); \fill (-2l) circle (3pt); \draw (-2l) node[shift=({0,0.35})] {$-2$};
  \coordinate (-3l) at (-\r,3*\rv/4); \fill (-3l) circle (3pt); \draw (-3l) node[right] {-$3$};
  \coordinate (0l) at (-\r,3*\rv/4+\rva); \fill (0l) circle (3pt); \draw (0l) node[shift=({-0.25,0.25})] {$0_1$};
  \draw[thick] (-3l)--(-2l)--(-1l)--(-3l)--(0l);
  \draw[thick] (0l) -- ++(-\ger,0) (-2l) -- ++(-\ger/3,0);
  \draw[ultra thin] (-1l) -- ++({-sqrt(3)/2*\ger},-\ger/2) (-1l) -- ++(-\ger/2,{-sqrt(3)/2*\ger});
  \draw[ultra thin] (-3l) -- ++({-sqrt(3)/2*\ger},\ger/2) (-3l) -- ++(-\ger/2,{sqrt(3)/2*\ger});
  \coordinate (i) at (-\r/2,\rva/2); \fill (i) circle (3pt); \draw (i) node[right] {$i$};
  \coordinate (j) at (-\r/2,\rva); \fill (j) circle (3pt); \draw (j) node[below] {$j$};
  \draw[ultra thin] (i) -- (0b) (j) -- (0l);
  \foreach \x in {0,1}{
    \coordinate (-3c\x) at (0,\rv/4+\x*\rva); \fill (-3c\x) circle (3pt); \draw (-3c\x) node[shift=({0.25,-0.25})] {-3};
    \coordinate (-1c\x) at (\rb,\rv/2+\x*\rva); \fill (-1c\x) circle (3pt); 
    \coordinate (-2c\x) at (-\rb,\rv/2+\x*\rva); \fill (-2c\x) circle (3pt); 
    \coordinate (0c\x) at (0,3*\rv/4+\x*\rva); \fill (0c\x) circle (3pt); \draw (0c\x) node[shift=({0.25,0.25})] {0};
    \draw[thick] (-1c\x)--(-3c\x)--(-2c\x)--(0c\x)--(-1c\x)--(-2c\x);
};
  \draw[thick] (0c0) -- ++(0,\ger) (-3c1) -- ++(0,-\ger); \draw (0,\rv/2+\rva/2) node {$\vdots$};
  \draw[thick] (0b) -- (-3c0) (0c1) -- (-3t);
  \draw[thick] (0b) .. controls (-3*\r/4,-\rv/4) .. (-1l);
  \draw[thick] (-1t) .. controls (-3*\r/4,3*\rv/2+\rva) .. (0l);
  \draw (-\r/2,3*\rva/2) node {$F_1$};
  \draw [line width=2.5pt, -to] (\rv/2-0.2,\rva-0.2) -- (3*\rv/4-0.2,\rva-0.2) ;
\end{scope}
\begin{scope}[shift={(3.4*\r,0)}]
%
  \coordinate (-3t) at (0,5*\rv/4+\rva); \fill (-3t) circle (3pt); \draw (-3t) node[shift=({0.25,-0.25})] {-$3'$};
  \coordinate (-1t) at (-\rb,3*\rv/2+\rva); \fill (-1t) circle (3pt); \draw (-1t) node[shift=({-0.4,-0.25})] {-$2$}; %
  \coordinate (-2t) at (\rb,3*\rv/2+\rva); \fill (-2t) circle (3pt); \draw (-2t) node[shift=({0,-0.35})] {-$1$};%
  \draw[thick] (-3t)--(-2t)--(-1t)--(-3t);
  \draw[ultra thin] (-1t) -- ++(\ger/3,-{sqrt(8)/3*\ger}) (-1t) -- ++(-\ger/3,-{sqrt(8)/3*\ger});
  \coordinate (0l) at (-\r,3*\rv/4+\rva); \fill (0l) circle (3pt); \draw (0l) node[shift=({-0.25,0.25})] {$0_1$};
  \coordinate (i) at (-\r/2,\rva/2); \fill (i) circle (3pt); \draw (i) node[right] {$i$};
  \coordinate (j) at (-\r/2,\rva); \fill (j) circle (3pt); \draw (j) node[below] {$j$};
  \draw[ultra thin] (j) -- (0l); 
  \foreach \x in {0,1}{
    \coordinate (-3c\x) at (0,\rv/4+\x*\rva); \fill (-3c\x) circle (3pt); \draw (-3c\x) node[shift=({0.25,-0.25})] {-3};
    \coordinate (-1c\x) at (\rb,\rv/2+\x*\rva); \fill (-1c\x) circle (3pt); 
    \coordinate (-2c\x) at (-\rb,\rv/2+\x*\rva); \fill (-2c\x) circle (3pt); 
    \coordinate (0c\x) at (0,3*\rv/4+\x*\rva); \fill (0c\x) circle (3pt); \draw (0c\x) node[shift=({0.25,0.25})] {0};
    \draw[thick] (-1c\x)--(-3c\x)--(-2c\x)--(0c\x)--(-1c\x)--(-2c\x);
};
  \draw[thick] (0c0) -- ++(0,\ger) (-3c1) -- ++(0,-\ger); \draw (0,\rv/2+\rva/2) node {$\vdots$};
  \draw[thick] (0c1) -- (-3t); 
  \draw[thick] (-1t) .. controls (-3*\r/4,3*\rv/2+\rva) .. (0l);
  \draw[thick, dashed](-2t) .. controls (-\r,7*\rv/4+\rva) .. (0l);
  \draw[thick, dashed](0l) .. controls (-\r,-\rv/4) .. (-3c0);
  \draw[ultra thin, dashed](0l) .. controls (-\r/2,-\rv/4) .. (i);
  \draw (-\r/2,3*\rva/2) node {$F_1$}; 
\end{scope}
\end{tikzpicture}
\caption{Constructing a necklace.}
\label{fig:bothin}
\end{figure}
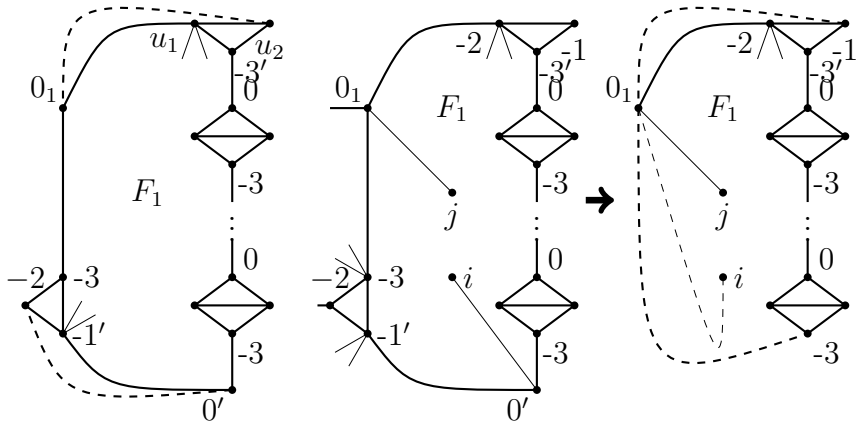

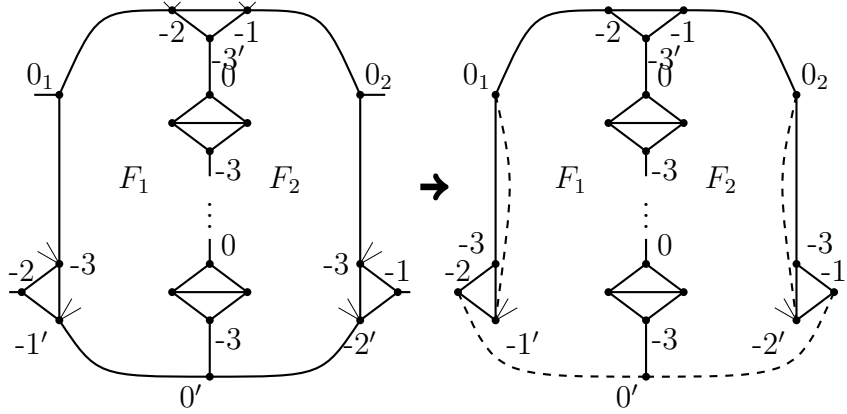
\begin{figure}[ht] 
\centering
\begin{tikzpicture}[scale=0.5]
\def \r {4}4.5
\def \rb {1}
\def \rv {3}
\def \rva {4.5}
\def \ger{2/3}
\begin{scope}
  \coordinate (0b) at (0,-\rv/4); \fill (0b) circle (3pt); \draw (0b) node[shift=({-0.25,-0.25})] {$0'$};
  \coordinate (-3t) at (0,5*\rv/4+\rva); \fill (-3t) circle (3pt); \draw (-3t) node[shift=({0.25,-0.25})] {-$3'$};
  \coordinate (-1t) at (-\rb,3*\rv/2+\rva); \fill (-1t) circle (3pt); \draw (-1t) node[below] {-$2$}; 
  \coordinate (-2t) at (\rb,3*\rv/2+\rva); \fill (-2t) circle (3pt); \draw (-2t) node[below] {-$1$};
  \draw[thick] (-3t)--(-2t)--(-1t)--(-3t);
  \draw[ultra thin] (-1t) -- ++(\ger/3,\ger/3) (-1t) -- ++(-\ger/3,\ger/3);
  \draw[ultra thin] (-2t) -- ++(\ger/3,\ger/3) (-2t) -- ++(-\ger/3,\ger/3);
  \coordinate (-1l) at (-\r,\rv/4); \fill (-1l) circle (3pt); \draw (-1l) node[below left] {-$1'$}; 
  \coordinate (-2l) at (-\r-\rb,\rv/2); \fill (-2l) circle (3pt); \draw (-2l) node[above] {-$2$}; 
  \coordinate (-3l) at (-\r,3*\rv/4); \fill (-3l) circle (3pt); \draw (-3l) node[right] {-$3$}; 
  \coordinate (0l) at (-\r,3*\rv/4+\rva); \fill (0l) circle (3pt); \draw (0l) node[shift=({-0.25,0.25})] {$0_1$};
  \draw[thick] (-3l)--(-2l)--(-1l)--(-3l)--(0l);
  \draw[thick] (0l) -- ++(-\ger,0) (-2l) -- ++(-\ger/2,0);
  \draw[ultra thin] (-1l) -- ++({sqrt(3)/2*\ger},\ger/2) (-1l) -- ++(\ger/2,{sqrt(3)/2*\ger});
  \draw[ultra thin] (-3l) -- ++({-sqrt(3)/2*\ger},\ger/2) (-3l) -- ++(-\ger/2,{sqrt(3)/2*\ger});
  \coordinate (-2r) at (\r,\rv/4); \fill (-2r) circle (3pt); \draw (-2r) node[below] {-$2'$}; 
  \coordinate (-1r) at (\r+\rb,\rv/2); \fill (-1r) circle (3pt); \draw (-1r) node[above] {-$1$}; 
  \coordinate (-3r) at (\r,3*\rv/4); \fill (-3r) circle (3pt); \draw (-3r) node[left] {-$3$}; 
  \coordinate (0r) at (\r,3*\rv/4+\rva); \fill (0r) circle (3pt); \draw (0r) node[shift=({0.25,0.25})] {$0_2$};
  \draw[thick] (-3r)--(-2r)--(-1r)--(-3r)--(0r);
  \draw[thick] (0r) -- ++(\ger,0) (-1r) -- ++(\ger/2,0);
  \draw[ultra thin] (-2r) -- ++({-sqrt(3)/2*\ger},\ger/2) (-2r) -- ++(-\ger/2,{sqrt(3)/2*\ger});
  \draw[ultra thin] (-3r) -- ++({sqrt(3)/2*\ger},\ger/2) (-3r) -- ++(\ger/2,{sqrt(3)/2*\ger});
  \foreach \x in {0,1}{
    \coordinate (-3c\x) at (0,\rv/4+\x*\rva); \fill (-3c\x) circle (3pt); \draw (-3c\x) node[shift=({0.25,-0.25})] {-3};
    \coordinate (-1c\x) at (\rb,\rv/2+\x*\rva); \fill (-1c\x) circle (3pt); 
    \coordinate (-2c\x) at (-\rb,\rv/2+\x*\rva); \fill (-2c\x) circle (3pt); 
    \coordinate (0c\x) at (0,3*\rv/4+\x*\rva); \fill (0c\x) circle (3pt); \draw (0c\x) node[shift=({0.25,0.25})] {0};
    \draw[thick] (-1c\x)--(-3c\x)--(-2c\x)--(0c\x)--(-1c\x)--(-2c\x);
};
  \draw[thick] (0c0) -- ++(0,\ger) (-3c1) -- ++(0,-\ger); \draw (0,\rv/2+\rva/2) node {$\vdots$};
  \draw[thick] (0b) -- (-3c0) (0c1) -- (-3t);
  \draw[thick] (0b) .. controls (-3*\r/4,-\rv/4) .. (-1l) (0b) .. controls (3*\r/4,-\rv/4) .. (-2r);
  \draw[thick] (-1t) .. controls (-3*\r/4,3*\rv/2+\rva) .. (0l) (-2t) .. controls (3*\r/4,3*\rv/2+\rva) .. (0r);
  \draw (-\r/2,\rva) node {$F_1$}; \draw (\r/2,\rva) node {$F_2$}; 
  \draw [line width=2.5pt, -to] (2*\rv-0.4,\rva-0.2) -- (9*\rv/4-0.4,\rva-0.2) ;
\end{scope}
\begin{scope}[shift={(2.9*\r,0)}]
  \coordinate (0b) at (0,-\rv/4); \fill (0b) circle (3pt); \draw (0b) node[shift=({-0.25,-0.25})] {$0'$};
  \coordinate (-3t) at (0,5*\rv/4+\rva); \fill (-3t) circle (3pt); \draw (-3t) node[shift=({0.25,-0.25})] {-$3'$};
  \coordinate (-1t) at (-\rb,3*\rv/2+\rva); \fill (-1t) circle (3pt); \draw (-1t) node[below] {-$2$}; 
  \coordinate (-2t) at (\rb,3*\rv/2+\rva); \fill (-2t) circle (3pt); \draw (-2t) node[below] {-$1$};
  \draw[thick] (-3t)--(-2t)--(-1t)--(-3t);
  \coordinate (-1l) at (-\r,\rv/4); \fill (-1l) circle (3pt); \draw (-1l) node[below right] {-$1'$}; 
  \coordinate (-2l) at (-\r-\rb,\rv/2); \fill (-2l) circle (3pt); \draw (-2l) node[above] {-$2$}; 
  \coordinate (-3l) at (-\r,3*\rv/4); \fill (-3l) circle (3pt); \draw (-3l) node[above left] {-$3$}; 
  \coordinate (0l) at (-\r,3*\rv/4+\rva); \fill (0l) circle (3pt); \draw (0l) node[shift=({-0.25,0.25})] {$0_1$};
  \draw[thick] (-3l)--(-2l)--(-1l)--(-3l)--(0l);
  \draw[ultra thin] (-1l) -- ++({sqrt(3)/2*\ger},\ger/2) (-1l) -- ++(\ger/2,{sqrt(3)/2*\ger});
  \coordinate (-2r) at (\r,\rv/4); \fill (-2r) circle (3pt); \draw (-2r) node[below left] {-$2'$}; 
  \coordinate (-1r) at (\r+\rb,\rv/2); \fill (-1r) circle (3pt); \draw (-1r) node[above] {-$1$}; 
  \coordinate (-3r) at (\r,3*\rv/4); \fill (-3r) circle (3pt); \draw (-3r) node[above right] {-$3$}; 
  \coordinate (0r) at (\r,3*\rv/4+\rva); \fill (0r) circle (3pt); \draw (0r) node[shift=({0.25,0.25})] {$0_2$};
  \draw[thick] (-3r)--(-2r)--(-1r)--(-3r)--(0r);
  \draw[ultra thin] (-2r) -- ++({-sqrt(3)/2*\ger},\ger/2) (-2r) -- ++(-\ger/2,{sqrt(3)/2*\ger});
  \foreach \x in {0,1}{
    \coordinate (-3c\x) at (0,\rv/4+\x*\rva); \fill (-3c\x) circle (3pt); \draw (-3c\x) node[shift=({0.25,-0.25})] {-3};
    \coordinate (-1c\x) at (\rb,\rv/2+\x*\rva); \fill (-1c\x) circle (3pt); 
    \coordinate (-2c\x) at (-\rb,\rv/2+\x*\rva); \fill (-2c\x) circle (3pt); 
    \coordinate (0c\x) at (0,3*\rv/4+\x*\rva); \fill (0c\x) circle (3pt); \draw (0c\x) node[shift=({0.25,0.25})] {0};
    \draw[thick] (-1c\x)--(-3c\x)--(-2c\x)--(0c\x)--(-1c\x)--(-2c\x);
};
  \draw[thick] (0c0) -- ++(0,\ger) (-3c1) -- ++(0,-\ger); \draw (0,\rv/2+\rva/2) node {$\vdots$};
  \draw[thick] (0b) -- (-3c0) (0c1) -- (-3t);
  \draw[thick, dashed] (0b) .. controls (-\r,-\rv/4) .. (-2l) (0b) .. controls (\r,-\rv/4) .. (-1r);
  \draw[thick, dashed] (0l) .. controls (-7*\r/8,\rva) .. (-1l) (0r) .. controls (7*\r/8,\rva) .. (-2r);
  \draw[thick] (-1t) .. controls (-3*\r/4,3*\rv/2+\rva) .. (0l) (-2t) .. controls (3*\r/4,3*\rv/2+\rva) .. (0r);
  \draw (-\r/2,\rva) node {$F_1$}; \draw (\r/2,\rva) node {$F_2$}; 
\end{scope}
\end{tikzpicture}
\caption{Constructing a semi-cover having two internal, non-triangular faces.}
\label{fig:inout}
\end{figure}

  We now proceed to the second case: $u_1=-1,\, u_2=-2$.

  First suppose that not all the paths $(-3,1,2)$ lying in the closure of $F_1 \cup F_2$ lie on the triangles $(-3,1,2)$. By Remark~\ref{rem:maxpath}, we must necessarily have $v_1=v_2=-3$, and moreover, at one of these vertices, the edge $(-3,1)$ goes in $F_1 \cup F_2$ and the edge $(-3,1)$ goes out, and at the other one, they go opposite way. Counting the edges $(-2,1)$ in $F_1$ and $(-1,2)$ in $F_2$ we obtain $a_1=a_2=b_2$. Counting the edges $(-3,1)$ in $F_1 \cup F_2$ we get $a_1+a_2=m+2$ which gives $b_1=a_1-2$. It follows that there are exactly $a_1$ vertices labelled $-1$ on the boundary of $F_1$, and so all four edges $(-1',2), (-1',3), (u_1,2)$ and $(u_1,3)$ go in $F_1$. Performing the surgery as on the left in Figure~\ref{fig:bothin} (with $u_1=-1,\, u_2=-2$) we obtain a necklace which is a contradiction with Corollary~\ref{c:nonecklace}\eqref{it:noneckl}.

  Now suppose that every path $(-3,1,2)$ lying in the closure of $F_1 \cup F_2$ lies on a triangle $(-3,1,2)$. By Remark~\ref{rem:maxpath} applied to $U=F_1$ and $i=2$, every path $(-2,1,3)$ lying in the closure of $F_1$ lies on a triangle $(-2,1,3)$, and similarly, every path $(-1,2,3)$ lying in the closure of $F_2$ lies on a triangle $(-1,2,3)$. It follows that every triangle $(1,2,3)$ lying in $F_i, i=1,2$, is a subgraph of a trapezium of type $i$, lying in the closure of $F_i$, with both vertices $i-3$ and $-3$ of the trapezium lying on the boundary of $F_i$. 

  We now use a simplified version of the argument in the last two paragraphs of the proof of Proposition~\ref{p:twofaces}. Consider the vertex $-3$ on the string lying immediately above the vertex $0'$ (in the notation of Figure~\ref{fig:sharem-2}). Without loss of generality we can assume that it is a vertex of a trapezium $T$ of type $2$ lying in (the closure of) the face $F_2$, with the vertices $-3$ and $-1$ lying on the boundary of $F_2$. From Lemma~\ref{l:support1}\eqref{it:trap-3top1} it follows that the vertex $-1$ of $T$ does not lie on the string, and so it can only be the vertex $v_2$. Then the path $(-3,1,-1,3)$ of $T$ lies above the paths $(-3,2,-1)$ , similar to how they are positioned in Figure~\ref{fig:trapat-3} (for otherwise we cannot connect both vertices $1$ and $3$ of $T$ to both $0'$ and $-2'$), with the edge $(2,0')$ forced. Then the vertex $1$ of $T$ cannot be connected to $0_2$, as we need an edge $(3,-2)$. Hence the vertex $1$ of $T$ is connected to a vertex $0$ lying on the string. Consider the vertex $v$ on the string lying on the boundary of $F_2$ immediately below that vertex $0$. Then $v$ is connected to a vertex labelled $3$ which is not a vertex of $T$. This vertex $3$ lies on a triangle $(1,2,3)$ which is supported on the string and whose top vertex of attachment to the string is either $0$ or $v \ne -3$ which contradicts Lemma~\ref{l:support1}\eqref{it:from0to-31}. This completes the proof in the case when both $F_1$ and $F_2$ are of length $3m$.

  Examining the above proof we can see that a contradiction is reached by considering only the part of $G'$ lying in the closure of $F_1 \cup F_2$. Hence in the case when one of the faces $F_i,\, i=1,2$, has length $3m$ and another one, $3(m-1)$, we can just add three vertices $0, i, j$ to the boundary of the shorter face $F_i$ (outside the string), assume that all the other edges of $G'$ incident to these vertices go out of $F_i$, and repeat the above argument. It should be noted however that one \emph{can} have $m-2$ beads on a boundary of a single face of length $3(m-1)$: it is not hard to construct an example of a $9$-face with two beads on its boundary and with a single triangle $(1,2,3)$ inside.
\end{proof}



\section{Proof of Theorem~\ref{th:nofold}}\label{sec:proofthm}

Recall that in the construction of the semi-cover $G'$ in Section~\ref{sec:pptH} we began with a cover $G$ of $K_{1,2, 2, 2}$ satisfying conditions~\eqref{it:minfold}, \eqref{it:max3}, \eqref{it:shortf} and \eqref{it:longf}. 
We then chose a long cycle $C$ covering an octahedral 3-cycle that contained no other such cycle in its interior.
The graph $G'$ and its subgraph $H$ were located inside the interior domain of $C$. 

Note that in fact $G$ must contain at least two such domains, one inside $C$ and one outside. (To see this we could re-embed $G$ with a face of $G'$ as the outer face).
The other may also be bounded by $C$, or may be bounded by another long cycle $C'$ covering an octahedral 3-cycle.
In any case, within each of these two domains there must be a planar semi-cover of $K_{1, 2, 2, 2}$ satisfying the conditions of Lemma~\ref{l:Hfaces}, and these two semi-covers are disjoint. 
In each semi-cover there is a subgraph $H$ covering a $K_4$ subgraph of $K_{1,2, 2, 2}$ that includes the vertex $0$.
Thus, if we show that $H$ has fold number at least $n$, this implies $G$ has fold number at least $2n$.

We can improve slightly on this bound by considering all long cycles in $G$ that cover an octahedral 3-cycle. We will refer to such cycles as \emph{long octahedral 3-cycles}. Firstly note that every vertex of $G$ not labelled $0$ lies on a long octahedral 3-cycle, otherwise it is surrounded by four short octahedral 3-cycles that induce a 4-cycle in its neighbourhood, and no neighbour labelled $0$ can be present. So if $G$ has fold number $n$ we can immediately see that the total length of all long octahedral 3-cycles is at least $6n$. In fact, we can say that it is strictly greater than $6n$ because inside a subcover $G'$ satsifying Lemma~\ref{l:Hfaces} there must be long octahedral 3-cycles present that intersect the surrounding long octahedral 3-cycle $C$ (otherwise $C$ is not minimal). So there are certainly some vertices that lie in two long octahedral 3-cycles.

In the other direction, we can bound the total length of long octachedral 3-cycles from above as follows.

\begin{lemma}\label{lem:longcycles}
Suppose $G$ has fold number $n$ and every possible $H$ satisfying Lemma~\ref{l:Hfaces} has fold number at least $h$ and contains at least $t$ triangles (with labels not appearing in $H$) inside its internal faces. Then the set of long octahedral $3$-cycles labelled $i, j, k$ or $-i,-j,-k$, where $\{|i|, |j|, |k|\} = \{1, 2, 3\}$, has total length at most $3(2n - 2h -2t)$. Thus the total length of all long octahedral $3$-cycles is at most $12(2n - 2h -2t)$    
\end{lemma} 

\begin{proof}
The set of long octahedral $3$-cycles labelled either $i, j, k$ or $-i,-j,-k$ (which are pairwise disjoint) bounds at least two domains that contain a subcover satisfying Lemma~\ref{l:Hfaces}.

The label $1$ appears at least $x+y$ times in these subcovers, for $x,y \in \{h,t\}$, and the label $-1$ appears at least $(h+t -x) + (h+t-y$) times. So in total they appear at least $2h + 2t$ times. Thus the remaining vertices labelled 1 or $-1$ that are not in these two subcovers number at most $2n - 2h - 2t$. Hence the total length of long cycles labelled $i, j, k$ or $-i,-j,-k$ is at most $3(2n - 2h - 2t)$. Since there are four pairs of disjoint faces in the octahedron, the final claim follows.
\end{proof}

To complete the proof of Theorem~\ref{th:nofold}, we will show that we can take the values $h=6$ and $t=3$ in Lemma~\ref{lem:longcycles}.
As observed above, this would immediately exclude the possibility that $G$ has fold number less than 12.
Now suppose $G$ is a 12-fold cover. Then it has 72 octahedral vertices, and long octahedral 3-cycles with total length at most $12(24-12-6)=72$, contradicting the observation that some vertex is in two such cycles.

When $H$ is an $h$ fold cover of $K_4$ with $3m$ vertices on the outer face, we have $3h - 2m$ internal octahedral vertices. Each has two neighbours in a triangle labelled $(1,2,3)$, and each vertex in these triangles has two octahedral neighbours in $H$. So the number of such triangles is at least $ \lceil \frac{2(3h-2m)}{2 \cdot 3} \rceil = \lceil h - \frac{2m}{3} \rceil$. When $h=6$, this is at least 3 unless $m=6$, in which case $H$ is a necklace, contradicting Corollary~\ref{c:nonecklace}\eqref{it:noneckl}. Thus we may take $t=3$ if $h=6$ in Lemma~\ref{lem:longcycles}.

It remains to show that there can be no $H$ subgraph satisfying Lemma~\ref{l:Hfaces} that has fold number $h\leq 5$.

By contracting all the triangles $(-1, -2, -3)$ in $H$ to single vertices and replacing strings of beads in $H$ with edges, we obtain a graph $H''$ which is connected, $3$-regular, bipartite and planar. Note that $H''$ may have double edges. 
If $a$ is the number of $0$ vertices in $H''$ then there are $2a$ vertices, $3a$ edges and $a+2$ faces. 
 It follows from Corollary~\ref{c:nonecklace}\eqref{it:noneckl} and Proposition~\ref{p:twofaces}  that $H$ has at least three internal, non-triangular faces. This means that $H''$ cannot be the graph with two vertices and three parallel edges, i.e. $a\geq 2$. 

Let $f_i$ be the number of $i$-faces in $H''$. Then the number of edges is $3a = \frac12\sum_i if_i$ and by Euler's formula we have $6=6a - 9a + 3\sum_i f_i = 3\sum_i f_i - \frac12\sum_i if_i$. This implies that 
\begin{equation}\label{eq:beads}
    2f_2 + f_4 = 6 + f_8 + 2f_{10} + 3f_{12} +\ldots
\end{equation}

Let $b$ be the number of beads in $H$, so the fold number is $h = a +b$.
We know by Lemma~\ref{l:Hfaces}\eqref{it:no6face} that every internal, non-triangular face of $H$ has length at least $9$. Thus any internal 2-face of $H''$ must have at least 2 beads on its boundary, and any internal 4-face of $H''$ needs one bead on its boundary. If the outer face is a 2-face it requires at least one bead, otherwise it corresponds to a non-facial 3-cycle in $G$, contradicting condition~\eqref{it:shortf}.
Since a single bead lies on only two faces, these requirements allow us to estimate the number of beads by double counting.
Moreover, it follows from Proposition~\ref{p:share} that two internal $2$- or $4$-faces of $H''$ cannot share all their beads.

If $f_2=0$ then $f_4\geq 6$ by~\eqref{eq:beads} and thus $a\geq 4$ (since the number of faces is $a+2$). There must be at least 5 internal 4-faces, and thus at least $5/2$ beads, so $h = a + b \geq 7$.

If $f_2 = 1$ then $f_4\geq 4$ by~\eqref{eq:beads} and $a\geq 3$. If the 2-face is internal then there are at least 3 internal 4-faces, giving $b\geq 5/2$. If the 2-face is external, then there are at least 4 internal 4-faces and all faces require at least one bead, giving $b\geq 5/2$. Thus $h\geq 6$. 

If $f_2\geq 3$ or $f_2=2$ and the outer face is not a 2-face, then we have two internal 2-faces and thus at least 4 beads because 2-faces cannot share an edge, so $h \geq 6$.

Finally, suppose $f_2=2$ and the outer face is a 2-face. Then we have at least 2 beads on the internal 2-face and another on the outer 2-face. If $a=2$ then $f_4=2$ and $H''$ is as shown in Figure~\ref{fig:3internalfaces}. It can be seen that in this case 4 beads are required to satisfy the requirement (Proposition~\ref{p:share}) that two internal faces do not share all their beads.

Thus in every case we have $h\geq 6$, completing the proof of Theorem~\ref{th:nofold}.

\begin{figure}[h]
    \centering
\begin{tikzpicture}[scale=0.8,rotate=90]
  \draw[thick] (-6, 0.5) .. controls (-6.7,1.5) .. (-6,2.9); 
  \draw[thick] (-6, 0.5) .. controls (-5.3,1.5) .. (-6,2.9);
 \draw[ thick ] (-6, 3.1)--(-6,4); 
  \draw[thick] (-6, 4) .. controls (-4,4) and (-4,-0.5) .. (-5.9,-0.55); 
    \draw[thick] (-6, 4) .. controls (-8,4) and (-8, -0.5) .. (-6.1,-0.55); 
 \draw[ thick ] (-6, 0.5)--(-6, -0.4);
 \draw[thick, fill=white] (-6, 3) circle (4pt) ; 
\filldraw[black] (-6, 4) circle (4pt) ;
\draw[thick, fill=white] (-6, -0.5) circle (4pt) ; 
\filldraw[black] (-6, 0.5) circle (4pt) ;
\end{tikzpicture}
    \caption{$H''$ that requires $4$ beads.}
    \label{fig:3internalfaces}
\end{figure}
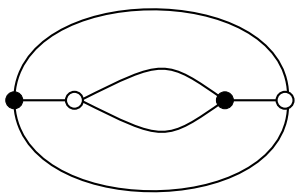

\end{document}